\documentclass[11pt,a4paper,english]{article}
\usepackage[utf8]{inputenc}
\usepackage{xcolor}
\usepackage{babel}
\usepackage{mathtools}
\usepackage{amsmath}
\usepackage{amsthm}
\usepackage{amssymb}
\usepackage{stmaryrd}
\usepackage[bookmarks=true,bookmarksnumbered=false,bookmarksopen=false,
 breaklinks=false,pdfborder={0 0 1},backref=false,colorlinks=false]
 {hyperref}
\hypersetup{
 linkcolor=blue}

\makeatletter

\pdfpageheight\paperheight
\pdfpagewidth\paperwidth

\numberwithin{equation}{section}
\numberwithin{figure}{section}

\usepackage{ dsfont }

\newcommand{\df}{\mathrm{d}}

\allowdisplaybreaks

\makeatother

\theoremstyle{plain}
\newtheorem{thm}{\protect\theoremname}[section]
\newtheorem{lem}[thm]{\protect\lemmaname}
\theoremstyle{remark}
\newtheorem{rem}[thm]{\protect\remarkname}
\theoremstyle{plain}
\newtheorem{prop}[thm]{\protect\propositionname}
\providecommand{\lemmaname}{Lemma}
\providecommand{\propositionname}{Proposition}
\providecommand{\remarkname}{Remark}
\providecommand{\theoremname}{Theorem}

\begin{document}
\global\long\def\df{\mathrm{def}}%
\global\long\def\eqdf{\stackrel{\df}{=}}%
\global\long\def\ep{\varepsilon}%
\global\long\def\ind{\mathds{1}}%

\title{Small eigenvalues of hyperbolic surfaces with many cusps}
\author{Will Hide and Joe Thomas}
\maketitle
\begin{abstract}
We study topological lower bounds on the number of small Laplacian
eigenvalues on hyperbolic surfaces. We show that for any $a>0$,
there exists $b>0$ such that if $g<an$ and $X$ has genus $g$ and
$n$ cusps, then $X$ has at least $b(2g+n-2)$ small eigenvalues.
This saturates the sharp upper bound $2g+n-2$ of Otal and Rosas up
to a constant multiplicative factor.

{\footnotesize\tableofcontents{}}{\footnotesize\par}
\end{abstract}

\section{Introduction}

Let $X$ be a connected, finite-area hyperbolic surface. $X$ is said
to be of signature $\left(g,n\right)$ if it is homeomorphic to a
genus $g$ surface with $n$ punctures. The spectrum of the Laplacian
$\Delta_{X}$ contained in $\Big[0,\frac{1}{4}\Big)$ is discrete,
consisting of finitely many eigenvalues of finite multiplicity, referred
to as $\textit{small eigenvalues}$. 

A series of works (see Section \ref{subsec:Related-work}) give topological
upper bounds on the number of small eigenvalues, culminating in the
following result of Otal and Rosas, which is sharp \cite{Bu.1977}.
\begin{thm}[\cite{Ot.Ro09}]
\label{thm:BMM}Any hyperbolic surface of signature $\left(g,n\right)$
has at most $2g+n-2$ small eigenvalues.
\end{thm}

In this paper we are interested in finding topological $\textit{lower bounds}$
on the number of small eigenvalues on surfaces. 

It is likely that, for example as predicted by Selberg's conjecture
\cite{Se1965}, there exist examples of hyperbolic surfaces of arbitrarily
large volume with no eigenvalues in $\left(0,\frac{1}{4}\right)$.
However, these surfaces have few cusps in comparison to their genus.
On the contrary, for surfaces with a lot of cusps, a result of Zograf
guarantees the existence of a positive small eigenvalue. 
\begin{thm}[\cite{Zo1987}]
\label{thm:Zograf}There exists a constant $c>0$ such that for any
hyperbolic surface $X$ of signature $\left(g,n\right)$, its spectral
gap $\lambda_{1}\left(X\right)$ satisfies
\[
\lambda_{1}\left(X\right)\leqslant c\frac{g+1}{n}.
\]
\end{thm}

The purpose of this article is to show that for a hyperbolic surface
of large volume and with many cusps, the number of small eigenvalues
is in fact always linear in its volume.
\begin{thm}
\label{thm:main-thm}Given $a>0$, there exists $b>0$ such that
if $g<an$ and $X$ has genus $g$ and $n$ cusps, then $X$ has at
least $b(2g+n-2)$ small eigenvalues.
\end{thm}

This saturates the upper bound of Otal and Rosas up to a constant
multiplicative factor.

\subsection{Related work\label{subsec:Related-work}}

An erroneous statement of McKean \cite{McK1972} stating that $\lambda_{1}>\frac{1}{4}$
initiated great interest in the number of small eigenvalues of hyperbolic
surfaces. Indeed, it was proven by Randol \cite{Ra.74} that given
any $N\in\mathbb{N},$ one can always construct a surface with at
least $N$ small eigenvalues. Buser \cite{Bu.1977} showed by a different
construction that in fact there exist surfaces of signature $\left(g,n\right)$
with $2g+n-2$ small eigenvalues and also that $\lambda_{4g-2}>\frac{1}{4}$.
This latter result was extended to surfaces with cusps by Zograf \cite{Zo1987}
with $4g-2$ replaced by $4g+4n-2$. In \cite{Sc1990}, Schmutz improved
on Buser's result by proving that $\lambda_{4g-3}>\frac{1}{4}$ and
for genus $2$ surfaces that $\lambda_{2}>\frac{1}{4}$ \cite{Sc1991}.
It was conjectured by Buser \cite{Bu.1977} and Schmutz \cite{Sc1990}
that $\lambda_{2g+n-2}>\frac{1}{4}$ and this was later proven by
Otal and Rosas \cite{Ot.Ro09} (c.f. Theorem \ref{thm:BMM}). This
was extended to surfaces of finite type by Ballmann, Matthiesen and
Mondal \cite{Ba.Ma.Mo17}.

To the best of our knowledge, the only recorded non-trivial topological
lower bound on the number of small eigenvalues is the aforementioned
result of Zograf \cite{Zo1987} c.f. Theorem \ref{thm:Zograf}.

Another natural line of study is to understand topological bounds
on the number of small eigenvalues for random surfaces. To this end,
there has been significant progress made in the Weil-Petersson random
model on the moduli space $\mathcal{M}_{g,n}$. For example, in a
previous work of the authors \cite{Hi.Th.23}, it was shown that for
any $0<\varepsilon<\frac{1}{4}$ there is an explicit constant $C(\varepsilon)>0$
such that for any fixed $g\geqslant0$, a Weil-Petersson random surface
in $\mathcal{M}_{g,n}$ with $n$ large has at least $C(\ep)\left(2g+n-2\right)$
eigenvalues below $\ep$ with probability tending to $1$ as $n\to\infty$.
It is natural to expect the same result for any $g=g(n)$ with $g=o(n)$.
 Interestingly, even in the regime $\sqrt{g}\ll n\ll g$, one is
still able to find a small non-zero eigenvalue on a Weil-Petersson
random surface in $\mathcal{M}_{g,n}$ with probability tending to
$1$ as $g\to\infty$ due to the work of Shen and Wu \cite{Sh.Wu22}. 

\subsection{A word on the proof}

The starting point for our method is a basic geometric observation.
We show, by a simple area argument, that given $C_{1}>0,$ there exist
$C_{2}$ and $C_{3}$ such that any surface of signature $\left(g,n\right)$
with $g<C_{1}n$ necessarily contains at least $C_{2}n$ disjoint
embedded pairs of pants with at least one cusp whose total boundary
length is bounded above by $C_{3}$.

Now if the constant $C_{3}$ were sufficiently small to ensure that
every such pair of pants as above had a Dirichlet eigenvalue below
$\frac{1}{4}$, then Theorem \ref{thm:main-thm} would follow from
a min-max argument. However this cannot be the case, which we explain
in Remark \ref{rem:dirichlet}. Instead, we exploit the fact that
any pair of pants $P$ with at least one cusp has critical exponent
$\delta\left(P\right)>\frac{1}{2}$ and, given a uniform upper bound
on the length of the boundaries of the $P$, one has $\delta\left(P\right)>\frac{1}{2}+\kappa$
for some $\kappa>0$ uniform over all of the pants. Our goal is then
to establish an effective geodesic counting result for these pants
and apply an argument using Selberg's trace formula (Section \ref{subsec:Selberg's-trace-formula})
to conclude.

Suppose first that the systoles of the surfaces are uniformly bounded
away from $0$. Then by an effective prime geodesic theorem following
Naud \cite{Na2005} (see Section \ref{sec:PGT}) it follows that the
number of closed geodesics on $X$ of length $\leqslant T$ grows
like
\[
\frac{n}{T}\left(e^{\left(\frac{1}{2}+\kappa\right)T}+O\left(e^{\left(\frac{1}{2}+\kappa\right)\frac{T}{2}+\frac{T}{4}}\right)\right).
\]
Because the exponential growth in $T$ is faster than $\exp\left(\frac{T}{2}\right)$,
these geodesics give a large contribution to the geometric term in
the trace formula which we aim to exploit. By extracting a count for
the number of small eigenvalues from the spectral side of the trace
formula and arguing that the remaining terms have little effect on
the contribution from the geometric term, we arrive at the result.

In general we need to work harder to account for short geodesics.
Although it is natural to expect that having more very short geodesics
would actually improve lower bounds on the number of small eigenvalues,
the possible existence of very short geodesics causes issues in our
trace formula argument. The reason is that it becomes difficult to
distinguish between the contribution of small eigenvalues from the
contribution of scattering poles and possibly embedded eigenvalues
accumulating close to $s=\frac{1}{2}$. 

To get around this issue, we use a comparison argument following \cite{Co.Co.1989}.
If a surface $X$ has a collection of very short geodesics $\left\{ \gamma_{i}\right\} $,
we compare the small eigenvalues of $X$ to the small eigenvalues
of the limiting surface $X_{\infty}$ obtained by shrinking each $\gamma_{i}$
to a node. By exploiting some geometric control over $X_{\infty}$
(see Lemma \ref{lem:curves-on-limit-surface}) we can apply the previous
arguments to guarantee that $X_{\infty}$ has many small eigenvalues.
Using a variational argument (Proposition \ref{prop:passing-evalues})
we conclude that $X$ satisfies the conclusion of the theorem. 

\subsection{Notation}

For functions $f,h:\mathbb{N}\to\mathbb{\mathbb{R}}$ we write $f=O\left(h\right)$
to mean there exists $c>0$ and $N_{0}$ such that $|f(n)|\leqslant ch(n)$
for all $n\geqslant N_{0}$. We refer to $c$ as the implied constant
and it will be independent of any other parameters unless specified
otherwise.

\section{Geometric preliminaries}

\subsection{Existence of many pants with controlled boundary length}

We begin by proving that a surface with many cusps has many embedded
pairs of pants bounding one or two cusps with good control over their
boundary length.
\begin{lem}
\label{thm:many-pants}For any $a>0$ there exist constants $c_{1},c_{2}>0$
such that for any $g,n$ with $g<an$, any hyperbolic surface $X$
of signature $\left(g,n\right)\neq(0,3)$ has at least $c_{1}n$ embedded
subsurfaces with geodesic boundary and disjoint interior, which are
pairs of pants with at least one cusp and with total boundary length
$\leqslant c_{2}$. In fact, we can take 
\[
c_{1}\geqslant\frac{1}{64}\exp\left(-\frac{1}{2}-10\pi(2a+1)-2\exp\left(10\pi(2a+1)\right)\right)
\]
and $c_{2}\leqslant10\pi(2a+1).$
\end{lem}

\begin{proof}
Note first that if $X$ is a once-punctured torus, then its systole
is bounded above by $2\mathrm{arcsinh(3)}$ and this forms the boundary
of the desired pants. Since $(g,n)\neq(0,3)$ we can now assume that
$X$ has Euler characteristic at most $-2$. The length $1$ horocycle
around each cusp provides disjoint neighbourhoods that are isometric
to the cuspidal domain 
\[
\langle z\mapsto z+1\rangle\backslash\left\{ z\in\mathbb{H}\mid\text{Im}z>1\right\} .
\]
Suppose that one continues to grow these neighbourhoods simultaneously
so that the horocycles have length $C$ and that there are still at
least $\frac{n}{2}$ disjoint regions that are isometric to the cuspidal
domain 
\[
\langle z\mapsto z+1\rangle\backslash\left\{ z\in\mathbb{H}\mid\text{Im}z>\frac{1}{C}\right\} .
\]
By comparing volumes, one must have 
\[
\text{Vol}\left(X\right)=2\pi\left(2g+n-2\right)\geqslant\frac{n}{2}C.
\]
However, since $g<an$, if $C>4\left(2a+1\right)\pi$ we obtain a
contradiction, and there must be at least $\frac{n}{2}$ cusps for
which either the horocycle self intersects or intersects another horocycle
at time $<C$.

\noindent\textbf{Claim: }If for some $t\leqslant C$, the horocycles
of length $t$ around two cusps intersect, and their time $t'$ horocycles
do not self-intersect for any $t'\leqslant t$, then there is a simple
closed geodesic of length $\leqslant2C$ that separates off a pair
of pants bounding the two cusps.
\begin{proof}[Proof of claim.]
 Let $t\leqslant C$ be the length at which the two horocycles $H^{1}_{t}$
and $H^{2}_{t}$ first intersect. Since this is the first time that
the horocycles intersect, the intersections are a discrete set of
tangent points. Pick one such point $z$ and let $x_{1}$ and $x_{2}$
be the corresponding points on the non-intersecting $t-\varepsilon$
horocycles $H^{1}_{t-\varepsilon}$ and $H^{2}_{t-\varepsilon}$,
and join them by a shortest geodesic arc $\alpha$. Their union $H^{1}_{t-\varepsilon}\cup H^{2}_{t-\varepsilon}\cup\alpha$
is a closed curve that bounds just the two cusps from the remainder
of the surface since the horocycles do not self-intersect before time
$t$.

A regular neighbourhood of $H^{1}_{t-\varepsilon}\cup H^{2}_{t-\varepsilon}\cup\alpha$
consists of three simple closed curves, two of which are freely homotopic
to the respective cusps and the third that bounds the two cusps. This
latter curve is freely homotopic to a simple closed geodesic that
forms a boundary of an embedded pair of pants with the two cusps.
It is not contractible since it separates off the two cusps and it
is not homotopic to a cusp since otherwise the surface would be a
pair of pants, which we have excluded. The length of this boundary
geodesic is at most $2(t-\varepsilon)+2\ell(\alpha)\leqslant2(t-\varepsilon)+4\log\left(1+\frac{\varepsilon}{t-\varepsilon}\right)$
since $\alpha$ has length equal to twice the distance between $H_{t}$
and $H_{t-\varepsilon}$ which, since they are horocycle neighbourhoods
isometric to the cuspidal domain, is $\log\left(1+\frac{\varepsilon}{t-\varepsilon}\right)$
by direct computation in $\mathbb{H}$. Taking $\varepsilon\to0$
gives a bound on the bounding geodesic as $2t\leqslant2C$.
\end{proof}

\noindent\textbf{Claim: }If a horocycle of length $<C$ self intersects,
then there are either two simple closed geodesics with total length
$\leqslant C$ that separate off a pair of pants bounding the cusp,
or there is a simple closed geodesic of length $\leqslant C$ bounding
the pair of pants containing the cusp and one other cusp. In the latter
case, the two cusps have horocycles of length $\leqslant C$ that
intersect.
\begin{proof}[Proof of claim.]
 Let $t\leqslant C$ be the length of the horocycle at which it first
self-intersects and let $z$ be a point of self-intersection. Given
$\varepsilon>0$ small, let $H_{t-\varepsilon}$ be the horocycle
of length $t-\varepsilon$ around the cusp so that $H_{t-\varepsilon}$
has no self-intersection. Let $\alpha$ and $\beta$ be two disjoint
geodesics that originate at $z$ and meet $H_{t-\varepsilon}$ perpendicularly,
dividing it into two pieces $H^{a}_{t-\varepsilon}\sqcup H^{b}_{t-\varepsilon}=H_{t-\varepsilon}$. 

The boundary of a regular neighbourhood of $\alpha\cup\beta\cup H_{t-\varepsilon}$
consists of three simple closed curves freely homotopic to $H_{t-\varepsilon}$,
$\alpha\cup\beta\cup H^{a}_{t-\varepsilon}$ and $\alpha\cup\beta\cup H^{b}_{t-\varepsilon}$
respectively. The latter two are freely homotopic to either a pair
of disjoint simple closed geodesics that bound a pair of pants with
the cusp, or one is freely homotopic to another cusp and the other
to a simple closed geodesic bounding a pants with the two cusps. Note
that we can't have both of these curves being homotopic to a cusp
otherwise $X$ would be a thrice punctured sphere.

The total boundary length of the resulting pair of pants is at most
$2\ell(\alpha)+2\ell(\beta)+t-\varepsilon\leqslant4\log\left(1+\frac{\varepsilon}{t-\varepsilon}\right)+t-\varepsilon$
for any $\varepsilon>0$. The bound on the lengths of $\alpha$ and
$\beta$ comes from the fact that they have length equal to the distance
between length $t$ and length $t-\varepsilon$ horocycles of the
same cusp which are isometric to cuspidal domains where the computation
is explicit in $\mathbb{H}$. Taking $\varepsilon\to0$ gives an upper
bound of $t\leqslant C$ as required.

It remains to prove that if the pants has two cusps, then their horoballs
with horocycles of length $\leqslant C$ intersect. To this end, let
$\mathfrak{a}$ be the original cusp and $\mathfrak{b}$ be the new
one. Let's assume that $\gamma=\alpha\cup\beta\cup H^{a}_{t-\varepsilon}$
is the closed curve freely homotopic to $\mathfrak{b}$ so that it
corresponds to a primitive parabolic element $g$ in the deck transformation
group of $X$. Fix a base point $x$ for $\gamma$ on the length 1
horocycle around $\mathfrak{a}$ and lift $\gamma$ to $\tilde{\gamma}$
in $\mathbb{H}$ starting at some lift $\tilde{x}$ of $x$ and ending
at $g\tilde{x}$. 

Let $H$ be a lift of the horoball of length $t$ around the cusp
$\mathfrak{a}$ that contains $\tilde{x}$. Thus, $\tilde{\gamma}$
consists of an arc from $\tilde{x}$ to $\partial H$ and then leaves
$H$ and joins up to $g\tilde{x}$ which then resides in a tangent
lift $H'$ to $H$ of the length $t$ horoball around $\mathfrak{a}$.
Since $g\tilde{x}$ lies on the image of the lifted length 1 horocycle
under $g$, we have $gH=H'$. Now conjugate $g$ so that it acts like
$z\mapsto z+L$ on $\mathbb{H}$ for some $L>0$ so that the cusp
$\mathfrak{b}$ is at $\infty$. The horoball $H$ conjugates to a
Euclidean circle tangent to the real-axis and $H'=gH$ is its horizontal
shift by $L$. Since $H'$ and $H$ are tangent, it follows that their
diameter must be $L$. By shifting we can assume then that $H$ is
subtended by $x=0$ and $x=L$ which contains a cuspidal domain around
$\mathfrak{b}$. Expanding a horocycle neighbourhood about $\mathfrak{b}=\infty$,
we see that it meets $H$ at $y=L$ which corresponds to a horocycle
of length $1$ on $X$ around $\mathfrak{b}$. Since it is tangent
to the horocycle of length $t$ around $\mathfrak{a}$, we conclude
that there exists some $1<t'\leqslant t\leqslant C$ for which the
length $t'$ horocycles around $\mathfrak{a}$ and $\mathfrak{b}$
meet as required.
\end{proof}

We now want to show that we can find at least $c_{1}n$ $\textit{disjoint}$
embedded pairs of pants with control over their boundary lengths.
Consider the graph defined as follows. Put a vertex for each cusp
whose length $t$ horocycle either self-intersects or has intersected
another length $t$ horocycle for some $t\leqslant C$, so that there
are $V\geqslant\frac{n}{2}$ vertices. Connect vertices by an edge
if their length $C$ horocycles intersect. 

We want to argue that there are at least $c_{1}n$ singleton vertices
or disjoint pairs of vertices connected by an edge. For this it is
helpful to have a bound on the maximal degree of the graph.

\noindent\textbf{Claim: }There is a constant $H$, depending only
on $C$, such that the maximal degree of the graph constructed is
at most $H\leqslant8\exp\left(\frac{1}{2}+2C+2\exp\left(2C\right)\right)$.
\begin{proof}
Take any cusp $\mathfrak{a}$ and any point $y_{\mathfrak{a}}$ on
the unique length $1$ horocycle around the cusp. Let $C_{\mathfrak{a}}$
be the length-$C$ horocycle around $\mathfrak{a}$. For each length-$C$
horocycle that intersects $C_{\mathfrak{a}}$, there is a different
geodesic loop based at $y_{\mathfrak{a}}$ with length bounded by
some constant dependent only upon $C$. Indeed, if $\mathfrak{b}$
is a cusp joined to $\mathfrak{a}$ in the graph, take the geodesic
loop based at $y_{\mathfrak{a}}$ that is homotopic to the curve that
joins $y_{\mathfrak{a}}$ to an intersection point of the length-$C$
horocycles about $\mathfrak{a}$ and $\mathfrak{b}$ (this has length
at most $e^{C}+\frac{C}{2}$), then follows the length $C$ horocycle
of $\mathfrak{b}$ around (adding a length of $C$) and returns along
the first curve to $y_{\mathfrak{a}}$ (adding another length of at
most $e^{C}+\frac{C}{2}$). The total length of the geodesic loop
is thus bounded by $2C+2e^{C}$.

The geodesic loops are non-contractible (with fixed base at $y_{\mathfrak{a}}$)
since they bound a cusp; they are homotopically distinct because they
bound distinct cusps. It follows that the number $H$ of length-$C$
horocycles which intersect $C_{\mathfrak{a}}$ is bounded above by
the number of geodesic loops based at $y_{\mathfrak{a}}$ with length
bounded above by $L=2C+2e^{C}$. Since $\text{InjRad}(y_{\mathfrak{a}})>\frac{1}{2}$,
an area argument gives an absolute bound for $H\leqslant8e^{L+\frac{1}{2}}$
(see for example the proof of \cite[Lemma 6.6.4]{Bu2010}).
\end{proof}

Now let $S$ be the collection of vertices with degree $\geqslant1$
and $x=|S|\leqslant V$, so that there are $V-x$ singletons. If $x=0$
we are done. Take the vertex $v$ of largest degree in $S$ and pick
one vertex $v'$ that it is connected to. Pair $v$ and $v'$ and
remove them from $S$ along with any of their neighbours. There are
at most $H-1$ neighbours of $v$ and $v'$ respectively removed from
$S$ at this step, so that with $v$ and $v'$, we remove at most
$2H$ elements from $S$ but obtain one pair. Repeat this process
until the set $S$ is empty. The process cannot terminate before $\lfloor\frac{x}{2H}\rfloor$
steps, giving at least $\lfloor\frac{x}{2H}\rfloor$ pairs. Then $V-x+\lfloor\frac{x}{2H}\rfloor>V-x\left(1-\frac{1}{2H}\right)-\frac{1}{2}\geqslant\frac{n}{8H}\geqslant c_{1}n$
for some constant $c_{1}>0$ providing $>c_{1}n$ singletons or pairs
with no vertex appearing in more than one pair. In fact, by the upper
bound on $H$ we can take $c_{1}>\frac{1}{64}\exp\left(-\frac{1}{2}-2C-2\exp\left(2C\right)\right)$.

Each singleton corresponds to a cusp whose length $t\leqslant C$
horocycle has self-intersected and does not intersect any other length
$t$ horocycle and so the second claim above gives a pair of pants
separating off this cusp with total boundary at most $C$. 

A pair corresponds to two cusps whose length $t\leqslant C$ horocycle
intersect. If there is a $t'\leqslant C$ such that the length $t'$
horocycle of one of the cusps self-intersects but for every $t\leqslant t',$
the length $t$ horocycles of the cusps do not intersect, then we
disregard the other cusp and consider the pair of pants constructed
from the self-intersection which has total boundary length at most
$C$ by the second claim. Note that if this involves another cusp,
then either that was the paired cusp or by construction it was a cusp
that we threw away when selecting the pairs. In any case, it is not
a cusp contained in any of the other pairs, nor a singleton.

If this is not the case, then there is a time $t\leqslant C$ for
which the length $t$ horocycles of the cusps first intersect and
for every $t'\leqslant t$, the length $t'$ horocycles around the
cusps do not self-intersect. By the first claim, there is a simple
closed geodesic of length at most $c_{2}\eqdf2C$ that separates off
a pair of pants bounding the two cusps.

In conclusion, there are at least $c_{1}n$ disjoint embedded pairs
of pants in $X$ bounding one or two cusps each with total boundary
length at most $c_{2}$ and with disjoint interiors. The stated bounds
on $c_{1}$ and $c_{2}$ now follow from taking $C=5\pi(2a+1)$. 
\end{proof}

\begin{rem}
\label{rem:dirichlet}At this point, if it were true that the pants
guaranteed by Lemma \ref{thm:many-pants} all had a Dirichlet eigenvalue
below $\frac{1}{4}$, then after extending the eigenfunction by $0$
and applying a min-max argument, Theorem \ref{thm:main-thm} would
follow immediately. We now argue why this cannot be the case. 

Taking some more care to track explicit constants in the above argument,
we need to consider intersecting horocycles of length at least $2\pi$.
In the case where horocycles of length $2\pi$ touch, one can show
the resulting pair of pants $P$ has boundary length $2\text{arcosh}\left(2\pi^{2}-1\right)\approx7.261...$.Take
two copies of the pants $P$ and glue without twist to obtain a genus
$0$ surface $S$ with $4$ cusps. By \cite{Ot.Ro09} c.f. Theorem
\ref{thm:BMM}, $S$ can only have $1$ non-zero small eigenvalue.
Any small Dirichlet eigenfunction on the pants $P$ extends to a non-constant
eigenfunction $\phi$ on $S$ with eigenvalue $<\frac{1}{4}$ which
is odd with respect to the isometric involution given by reflection
in the boundary of $P$ in $S$. The surface $S$ also has another
isometric involution by reflection around the closed geodesic orthogonal
to the boundary of $P$ in $S$, decomposing $S$ as another doubled
pair of pants $P'$. One can calculate (using \cite[Formula 2.6.17]{Th2022})
that $P'$ has boundary length 
\[
4\text{arcsinh}\left(\frac{1}{\sinh\left(\frac{\text{arcosh}\left(2\pi^{2}-1\right)}{2}\right)}\right)\approx1.3191...<2\text{arcosh}\left(2\pi^{2}-1\right).
\]
 If it were true that every pair of pants with two cusps and with
boundary of length $\leqslant2\text{arcosh}\left(2\pi^{2}-1\right)$
has a small Dirichlet eigenvalue, then both $P$ and $P'$ would also.
Extending the non-constant eigenfunction on $P'$ to $\phi'$ on $S$
would give that $\phi=\phi'$ but then $\phi$ vanishes on a pair
of simple closed geodesics intersecting at right angles and has to
have at least $4$ nodal domains. This contradicts Courant's nodal
domain theorem. Thus at least one of $P$ or $P'$ has no small Dirichlet
eigenvalue.
\end{rem}

\subsection{Prime Geodesic Theorem in pants\label{sec:PGT}}

To control the number of closed geodesics in each pair of pants, we
will make use of the following explicit prime geodesic theorem due
to Naud \cite{Na2005} with some suitable modifications. Let $P(\ell_{1},\ell_{2},\ell_{3})$
denote the pair of pants with geodesic boundary lengths $\ell_{1},\ell_{2},\ell_{3}$
where the boundary is a cusp if the length is zero. We write $N\left(P,T\right)$
to denote the number of primitive closed geodesics on $P$ of length
at most $T$ and $\delta\left(P\left(\ell_{1},\ell_{2},\ell_{3}\right)\right)$
be the critical exponent of $P\left(\ell_{1},\ell_{2},\ell_{3}\right)$.
\begin{thm}[\cite{Na2005}]
\label{thm:Naud-PGT}Let $A>B>0$ be given and let $P$ denote the
pair of pants with geodesic boundary lengths $\ell_{1},\ell_{2}\in[B,A]$
and a single cusp or a pair of pants with geodesic boundary length
$\ell_{1}\in[B,A]$ and two cusps. We have that
\[
N\left(P,T\right)=\mathrm{li}\left(e^{\delta(P)T}\right)+O\left(e^{\left(\frac{\delta(P)}{2}+\frac{1}{4}\right)T}\right),
\]
where the implied constant depends only on $A,B$.
\end{thm}

\begin{rem}
Here $\mathrm{li}(x)=\int^{x}_{0}\frac{\mathrm{d}t}{\log(t)}$ and
it satisfies the bounds 
\begin{equation}
\frac{e^{x}}{x}<\mathrm{li}\left(e^{x}\right)<\frac{e^{x}}{x}\left(1+\frac{2}{x}\right),\label{eq:log-int}
\end{equation}
for any $x\geqslant2$.
\end{rem}

\begin{proof}
In both cases, this is essentially \cite[Theorem 1.2]{Na2005} with
some small adaptations to make the error term uniform in $\ell_{1},\ell_{2}$
over the compact window. For the pants with two cusps, this was carried
out in \cite[Theorem 6.4]{Hi.Th.23}, and so we will give the details
for the case of one cusp. Let $P(\ell_{1},\ell_{2},0)$ be such a
pants with one cusp and two geodesic boundaries with lengths in $[B,A]$.

Note firstly that, for example by \cite[Theorem 3.1]{McM1998} \cite[Theorem 1, Theorem 5]{Po.Vy.17},
the Hausdorff dimension, and hence the rightmost resonance $\delta\left(P(\ell_{1},\ell_{2},0)\right)$
is continuous in $\ell_{1},\ell_{2}$. Moreover, we fix a topological
surface $\Sigma$ and can equip each $P(\ell_{1},\ell_{2},0)$ with
a marking $\varphi:\Sigma\to P(\ell_{1},\ell_{2},0)$ such that for
any free homotopy class of closed curve $[\alpha]$ on $\Sigma$,
the length $\ell_{[\alpha]}(P(\ell_{1},\ell_{2},0))$ of the geodesic
representative of $[\varphi(\alpha)]$ on $P(\ell_{1},\ell_{2},0)$
is continuous in $\ell_{1},\ell_{2}$. 

We now follow precisely the notation as in the proof of \cite[Theorem 1.2]{Na2005}.
Note that for fixed $X$ and $Y$ as in the proof, the number of closed
geodesics on $S_{\ell}$ can be bounded uniformly as $\ell_{1},\ell_{2}$
ranges over $[B,A]$ since the systole of these surfaces is uniformly
bounded below by $B$. It follows that in {[}ibid. Equation (2.19){]}
the geometric (right-hand) side of the formula is locally continuous
and hence continuous on $[B,A]$. Coupled with the continuity of the
first resonance, we see that the error term 

\[
\sum_{s\in\mathcal{R}_{S_{\ell}}\backslash\delta\left(P(\ell_{1},\ell_{2},0)\right)}\hat{g}(s)
\]
is also continuous (note that for us, $\mathcal{R}_{P(\ell_{1},\ell_{2},0)}=\mathcal{R}^{+}_{P(\ell_{1},\ell_{2},0)}\cup\left\{ \delta\left(P(\ell_{1},\ell_{2},0)\right)\right\} $
since by \cite{Ba.Ma.Mo17}, $\delta\left(P(\ell_{1},\ell_{2},0)\right)$
is the only resonance with $\textup{Re}\left(s\right)>\frac{1}{2}$).
The proof now continues identically up to {[}ibid. Equations (2.27)
and (2.28){]} where the bound on the error term is dependent on $\ell_{1},\ell_{2}$
in a pointwise manner of the form:
\[
\left|\sum_{s\in\mathcal{R}_{S_{\ell}}\backslash\delta_{0}(S_{\ell})}\hat{g}(s)\right|\leq C_{1}(\ell_{1},\ell_{2})+C_{2}(\ell_{1},\ell_{2})X^{a_{1}}Y^{b_{1}}+C_{3}(\ell_{1},\ell_{2})X^{c_{1}}Y^{d_{1}}.
\]
But, using the continuity of the left-hand side, the $C_{i}(\ell_{1},\ell_{2})$
on the right-hand side can be made uniform over all $\ell_{1},\ell_{2}\in[B,A]$.
 The proof then continues identically as in \cite[Theorem 1.2]{Na2005}
propagating these uniform bounds throughout the remainder.
\end{proof}

We will also need an analogous result for thrice punctured spheres
which is a consequence of the usual effective prime geodesic theorem
for finite-area geometrically finite hyperbolic surfaces and that
fact that the thrice punctured sphere has no non-trivial eigenvalues
below $\frac{1}{4}$.
\begin{thm}
\label{thm:thrice-punc-sphere}\cite[Theorem 1.1]{Na2005} Let $P$
be a thrice punctured sphere. Then, 
\[
N(P,T)=\mathrm{li}\left(e^{T}\right)+O\left(e^{\frac{3}{4}T}\right),
\]
with the implied constant universal.
\end{thm}

We also have the following lower bound on the critical exponents of
pants.
\begin{lem}
\label{lem:crit-exponent} For any $A>0$, there exists a constant
$\kappa(A)>0$ such that 
\[
\delta\left(P(\ell_{1},\ell_{2},0)\right)>\frac{1}{2}+\kappa(A),
\]
for any $\ell_{1},\ell_{2}\in[0,A]$.
\end{lem}

\begin{proof}
By e.g. \cite[Theorem 3.1]{McM1998} \cite[Theorem 1, Theorem 5]{Po.Vy.17},
$\delta\left(P\left(\ell_{1},\ell_{2},0\right)\right)$ is a continuous
function for $\ell_{1},\ell_{2}\in[0,A]$ so it attains its infimum
which is $>\frac{1}{2}$ by a result of Beardon \cite{Be.1968}.
\end{proof}

\section{Analytic Preliminaries}

\subsection{Selberg's trace formula\label{subsec:Selberg's-trace-formula}}

We recall Selberg's trace formula for a non-compact finite-area hyperbolic
surface which can for example be found in \cite[Theorem 10.2]{Iw2021}.
\begin{thm}
\label{thm:trace-formula}Let $X$ be a complete and oriented hyperbolic
surface with genus $g$ and $n$ cusps. Let 
\[
0=\lambda_{0}\leqslant\lambda_{1}\leqslant\lambda_{2}\leqslant\ldots
\]
denote the discrete spectrum of the (positive) Laplacian on $X$ and
for $j\geqslant0$, define
\[
r_{j}\eqdf\begin{cases}
\sqrt{\lambda_{j}-\frac{1}{4}} & \text{if }\lambda_{j}\geqslant\frac{1}{4},\\
i\sqrt{\frac{1}{4}-\lambda_{j}} & \text{if }\lambda_{j}<\frac{1}{4}.
\end{cases}
\]
Suppose that $f\in C^{\infty}_{c}(\mathbb{R})$ is even with Fourier
transform $h$, then
\begin{align}
\sum_{j\geqslant0}h(r_{j})+\frac{1}{4\pi} & \int^{\infty}_{-\infty}h(r)\frac{-\varphi'}{\varphi}\left(\frac{1}{2}+ir\right)\mathrm{d}r\nonumber \\
= & \frac{\mathrm{Vol}(X)}{4\pi}\int^{\infty}_{-\infty}h(r)r\tanh(\pi r)\mathrm{d}r\nonumber \\
 & +\sum_{\gamma\in\mathcal{P}\left(X\right)}\sum^{\infty}_{k=1}\frac{\ell(\gamma)f(k\ell(\gamma))}{2\sinh\left(\frac{k\ell(\gamma)}{2}\right)}+\frac{h(0)}{4}\mathrm{Tr}\left(I-\Phi\left(\frac{1}{2}\right)\right)\label{eq:trace-form}\\
 & -nf(0)\log(2)-\frac{n}{2\pi}\int^{\infty}_{-\infty}h(r)\psi(1+ir)\mathrm{\mathrm{d}}r,\nonumber 
\end{align}

where 
\begin{enumerate}
\item $\Phi(s)$ is the scattering matrix of $X$ at parameter $s$,
\item $\varphi(s)$ is the determinant of $\Phi(s)$,
\item $\psi(x)$ is the digamma function given by $\frac{\mathrm{d}}{\mathrm{d}x}\log\Gamma(x)$.
\end{enumerate}
\end{thm}

\noindent\textbf{Choice of function for the trace formula.}

\noindent We let $f_{1}\in C^{\infty}_{c}(\mathbb{R})$ be a real-valued,
even and non-negative function satisfying
\begin{enumerate}
\item $f_{1}$ is compactly supported on $[-1,1]$ and strictly positive
on $(-1,1)$,
\item $\hat{f_{1}}$ is non-negative on $\mathbb{R}\cup i\mathbb{R}$,
\item $\hat{f_{1}}$ is monotone decreasing on $[0,\infty)$.
\end{enumerate}
A test function satisfying the first two properties has previously
been used in several works studying $\lambda_{1}$ on hyperbolic surfaces
stemming from \cite{Ma.Na.Pu2022}. The existence of $f_{1}$ with
the additional property 3 is obtained for example in \cite{Tl2022}
where also there is control over the rate of decay of $\hat{f_{1}}$.
For completeness, we isolate and streamline the proof of the existence
of $f_{1}$. 
\begin{lem}
There exists a function $f_{1}$ with the above properties.
\end{lem}

\begin{proof}
Choose a function $\varphi\in C^{\infty}_{c}(\mathbb{R})$ that is
even, non-negative, compactly supported on $[-\frac{1}{2},\frac{1}{2}]$,
strictly positive on $(-\frac{1}{2},\frac{1}{2})$ and $\varphi'(x)<0$
on $(0,\frac{1}{2})$ -- for example choose a standard bump function.
Then, $\varphi*\varphi\in C^{\infty}_{c}(\mathbb{R})$ is even, non-negative
and compactly-supported on $[-1,1]$ with Fourier transform that is
even and non-negative on $\mathbb{R}\cup i\mathbb{R}$. 

Next define $\psi(x)=i(\varphi*\varphi)'(x)$ so that by properties
of the Fourier transform, $\hat{\psi}(\xi)=-\xi\hat{\varphi}(\xi)^{2}$.
It follows that $\hat{\psi}$ is odd and non-positive on $[0,\infty)$.
Since $\varphi*\varphi$ is smooth and compactly supported on $[-1,1]$,
Paley-Wiener theory allows $\hat{\psi}$ to be extended to an entire
function such that for all $N\in\mathbb{N}$ there exists $C_{N}>0$
with 
\begin{equation}
|\hat{\psi}(z)|\leqslant C_{N}\frac{e^{|\mathrm{Im}(z)|}}{(1+|z|)^{N}}.\label{eq:pw-bd}
\end{equation}
Now set 
\[
I=-\int^{\infty}_{0}\hat{\psi}(t)\mathrm{d}t
\]
and
\[
g(\xi)=I+\int_{\gamma}\hat{\psi}(z)\mathrm{d}z,
\]
where $\gamma$ is the contour formed of the union of the straight
line $\gamma_{1}$ from $0$ to $\mathrm{Re}(\xi)$ and the straight
line $\gamma_{2}$ from $\mathrm{Re}(\xi)$ to $\xi$, so that $g$
is entire and by the fundamental theorem of contour integration, $g'(\xi)=\hat{\psi}(\xi)$.
If $x\in\mathbb{R}$ then 

\[
g(x)=-\int^{\infty}_{x}\hat{\psi}(t)\mathrm{d}t,
\]
and since $\hat{\psi}|_{\mathbb{R}}$ is a Schwartz function we immediately
see that $g|_{\mathbb{R}}$ is also a Schwartz function. Using trivial
estimates of the integral of $\hat{\psi}$ over the contour $\gamma$
arising from (\ref{eq:pw-bd}) we also obtain the estimate

\[
|g(\xi)|\leqslant C_{2}(1+|\xi|)e^{|\mathrm{Im}(\xi)|}.
\]
Since $g$ is entire, the Paley-Wiener-Schwartz theorem tells us that
$g$ is the Fourier(-Laplace) transform of a distribution $f$ with
support in $[-1,1]$. But since $g|_{\mathbb{R}}$ is a Schwartz function
it follows also that $f$ must be a Schwartz function. 

In conclusion, we have constructed a function $f\in C^{\infty}_{c}(\mathbb{R})$
that is compactly supported in $[-1,1]$, that is even and real-valued
and since $\hat{f}'(x)=g'(x)=\hat{\psi}(x)\leqslant0$ on $[0,\infty)$,
its Fourier transform is monotone decreasing on $[0,\infty)$. The
evenness of $f$ can be verified by
\[
f(x)=\check{\hat{f}}(x)=\int_{\mathbb{R}}\hat{f}(r)e^{irx}\mathrm{d}r=\int_{\mathbb{R}}\hat{f}(r)e^{-irx}\mathrm{d}r=f(-x),
\]
where the third equality is due to evenness of $\hat{f}=g$ which
follows from oddness of $\hat{\psi}$. The fact that $f$ is real-valued
on $\mathbb{R}$ is immediate because $\hat{f}$ is real and even
so that

\[
f(x)=2\int^{\infty}_{0}\hat{f}(r)\cos(rx)\mathrm{d}r\in\mathbb{R}.
\]
The fact that $\hat{f}\geqslant0$ on $\mathbb{R}\cup i\mathbb{R}$
is also immediate. Indeed, if $x\geqslant0$ 

\[
\hat{f}(x)=g(x)=-\int^{\infty}_{x}\hat{\psi}(t)\mathrm{d}t\geqslant0,
\]
as $\hat{\psi}$ is non-positive on $[0,\infty)$, and evenness of
$g$ gives the conclusion for $x<0$. For $x=it$ with $t\geqslant0$
then because $\hat{\psi}(\xi)=-\xi\hat{\varphi}(\xi)^{2}$ and $\hat{\varphi}(\xi)^{2}\geqslant0$
on $\mathbb{R}\cup i\mathbb{R}$, we have
\[
\hat{f}(x)=I+i\int^{t}_{0}\hat{\psi}(is)\mathrm{d}s=I+\int^{t}_{0}s\hat{\varphi}(is)^{2}\mathrm{d}s\geqslant0
\]
again as $\hat{\psi}$ is non-positive on $[0,\infty)$, and similarly
for $t<0$. If we can further show that $f$ is strictly positive
on $(-1,1)$ then $f_{1}=f$ will satisfy all of the desired properties
of the lemma. To this end, by properties of the Fourier transform
and definition of $f$ and $\psi$, we have 
\[
\widehat{xf}(\xi)=i\hat{f}'(\xi)=i\hat{\psi}(\xi)=-\widehat{(\varphi*\varphi)'}(\xi).
\]
Taking inverse Fourier transforms results in

\[
f(x)=-\frac{(\varphi*\varphi)'(x)}{x},
\]
for $x\neq0$. At $x=0$ we can compute using evenness of $\varphi$

\begin{align*}
(\varphi*\varphi)''(0) & =\int_{\mathbb{R}}\varphi''(t)\varphi(t)\mathrm{d}t=-\int_{\mathbb{R}}\varphi'(t)^{2}\mathrm{dt}<0,
\end{align*}
and since $f(0)=-(\varphi*\varphi)''(0)$ we obtain the claim at $x=0$.
For $0<x<1$, we use integration by parts and the support and evenness
of $\varphi$ to write

\begin{align*}
(\varphi*\varphi)'(x) & =\int_{\mathbb{R}}\varphi(x-t)\varphi'(t)\mathrm{d}t=\int^{\frac{1}{2}}_{0}(\varphi(x-t)-\varphi(x+t))\varphi'(t)\mathrm{d}t.
\end{align*}
But, since $\varphi$ is even and strictly decreasing on $(0,\infty)$
and because $|x-t|<x+t$ for $x,t>0$ we have $\varphi(x-t)=\varphi(|x-t|)>\varphi(x+t)$.
Thus as $\varphi'(t)<0$ on $(0,\frac{1}{2})$ we conclude that $(\varphi*\varphi)'(x)<0$
for $0<x<1$ and it follows that $f(x)>0$ for $0<x<1$. Combining
this with evenness of $f$ results in the desired conclusion.
\end{proof}

\noindent For $T>0,$ we set

\begin{equation}
f_{T}(x)\eqdf f_{1}\left(\frac{x}{T}\right),\label{eq:fn-choice}
\end{equation}
and denote by $h_{T}$ the Fourier transform of $f_{T}$. In particular, 

\begin{equation}
h_{T}(r)=Th_{1}(Tr).\label{eq:h_t}
\end{equation}

\subsection{Scattering matrix}

We briefly recall some relevant information about the scattering matrix
$\Phi\left(s\right)$, where we follow \cite{Iw2021}.

Let $X=\Gamma_{X}\backslash\mathbb{H}$ where $\Gamma_{X}<\text{PSL}_{2}\left(\mathbb{R}\right)$.
We write $\mathcal{F}$ to denote a Dirichlet fundamental domain for
$\Gamma_{X}$. Since $\mathcal{F}$ is a non-compact polygon, it has
some of its vertices on $\mathbb{R}\cup\infty$ in $\mathbb{H\cup\partial}\mathbb{H}$
which we call cuspidal vertices. By e.g. \cite[Proposition 2.4]{Iw2021},
we can ensure that the cuspidal vertices are distinct modulo $\Gamma_{X}$.
The sides of $\mathcal{F}$ can be arranged in pairs so that the side
pairing motions generate $\Gamma_{X}$. The two sides of $\mathcal{F}$
meeting at a cuspidal vertex have to be pairs since the cuspidal vertices
are distinct modulo $\Gamma_{X}$. The side-pairing motion has to
fix the vertex and is therefore a parabolic element of $\Gamma_{X}$.
This gives rise to a cusp in the quotient $\Gamma_{X}\backslash\mathbb{H}$
and each cuspidal vertex corresponds to a unique cusp in this way.
We label the cuspidal vertices by $\mathfrak{a}_{1},...,\mathfrak{a}_{n}\in\text{Cusp}\left(X\right)$.
We denote the stabilizer subgroup of the vertex $\mathfrak{a}_{i}$
by
\[
\Gamma_{\mathfrak{a}_{i}}\stackrel{\text{def}}{=}\{\gamma\in\Gamma_{X}:\gamma\mathfrak{a}_{i}=\mathfrak{a}_{i}\}.
\]
Each $\Gamma_{\mathfrak{a}_{i}}$ is an infinite cyclic group generated
by the parabolic element $\gamma_{\mathfrak{a}_{i}}$, which is the
side-pairing motion at the vertex $\mathfrak{a}_{i}$. There exists
$\sigma_{\mathfrak{a}_{i}}\in\text{SL}_{2}\left(\mathbb{R}\right)$
such that 
\[
\sigma^{-1}_{\mathfrak{a}_{i}}\gamma_{\mathfrak{a}_{i}}\sigma_{\mathfrak{a}_{i}}=\begin{pmatrix}1 & 1\\
0 & 1
\end{pmatrix}.
\]
 $\sigma_{\mathfrak{a}_{i}}$ is determined up to right multiplication
by a translation. We choose $\sigma_{\mathfrak{a}_{i}}$ so that for
each $l\geqslant1$, the semi-strip 
\[
P\left(l\right)\stackrel{\text{def}}{=}\{z\in\mathbb{H}:0<x<1,y\geqslant l\},
\]
is mapped into $\mathcal{F}$ by $\sigma_{\mathfrak{a}_{i}}$. 

For $\mathfrak{a},\mathfrak{b}\in\text{Cusp}\left(X\right)$, we define
\[
\mathcal{C}_{\mathfrak{a}\mathfrak{b}}\eqdf\left\{ c>0:\exists\left(\begin{array}{cc}
* & *\\
c & *
\end{array}\right)\in\sigma^{-1}_{\mathfrak{a}}\Gamma\sigma_{\mathfrak{b}}\right\} .
\]
and the Kloosterman sum
\[
\mathcal{S}_{\mathfrak{a}\mathfrak{b}}\left(0,0,c\right)\eqdf\left|\left\{ d\mod c:\exists\left(\begin{array}{cc}
* & *\\
c & d
\end{array}\right)\in\sigma^{-1}_{\mathfrak{a}}\Gamma\sigma_{\mathfrak{b}}\right\} \right|.
\]
For $\text{Re}(s)>1$, the scattering matrix $\Phi(s)=\left(\varphi_{\mathfrak{a}\mathfrak{b}}(s)\right)_{\mathfrak{a,\mathfrak{b}}\in\text{Cusp}\left(X\right)}$
is given by a Dirichlet series 
\begin{equation}
\varphi_{\mathfrak{a}\mathfrak{b}}(s)=\sqrt{\pi}\frac{\Gamma\left(s-\frac{1}{2}\right)}{\Gamma\left(s\right)}\sum_{c>0}\mathcal{S}_{\mathfrak{a}\mathfrak{b}}\left(0,0,c\right)c^{-2s},\label{eq:dirichlet-matrix}
\end{equation}
where $c$ runs over the $c>0$ appearing in $\mathcal{C}_{\mathfrak{a}\mathfrak{b}}$.
The scattering determinant $\varphi(s)$ for $\text{Re}(s)>1$ is
given by a Dirichlet series
\begin{equation}
\varphi(s)=\left(\sqrt{\pi}\frac{\Gamma\left(s-\frac{1}{2}\right)}{\Gamma\left(s\right)}\right)^{n}\sum^{\infty}_{n=1}a_{n}b^{-2s}_{n}\label{eq:Dirichlet-determinant}
\end{equation}
with $a_{1}\neq0$ and $0<b_{1}<b_{2}<\dots<b_{n}\to\infty$ and this
converges absolutely for $\mathrm{Re}(s)\geq1+\varepsilon$ \cite[Page 160]{Iw2021}.

We record a trivial lower bound for the coefficient $b_{1}$.
\begin{lem}
\label{lem:dirichlet-coef}With notation as above, $b_{1}\geqslant1$.
\end{lem}

\begin{proof}
We first claim that for each $\mathfrak{a},\mathfrak{b}$, we have
$\min\mathcal{C}_{\mathfrak{a}\mathfrak{b}}\geqslant1$. By e.g. \cite[Page 50]{Iw2021},
we have that $\min\mathcal{C}_{\mathfrak{a}\mathfrak{b}}\geqslant\sqrt{\mathrm{max}\left\{ c_{\mathfrak{a}},c_{\mathfrak{b}}\right\} }$
where
\begin{equation}
c_{\mathfrak{a}}=\min\mathcal{C}_{\mathfrak{a}\mathfrak{a}}=\min\left\{ c>0:\left(\begin{array}{cc}
* & *\\
c & *
\end{array}\right)\in\sigma^{-1}_{\mathfrak{a}}\Gamma\sigma_{\mathfrak{a}}\right\} .\label{eq:def-c_a}
\end{equation}
Given 
\[
g=\left(\begin{array}{cc}
a & b\\
c & d
\end{array}\right)\in\sigma^{-1}_{\mathfrak{a}}\Gamma\sigma_{\mathfrak{a}},
\]
the isometric circle $C_{g}$ is the locus of points in $\mathbb{C}$
on which $g$ acts as a Euclidean isometry, i.e. the set of $z\in\mathbb{C}$
for which $\left|cz+d\right|=1$ which is a circle of radius $\frac{1}{|c|}$
centered at $-\frac{d}{c}$. Let 
\[
F_{\infty}\eqdf\left\{ x+iy\in\mathbb{H}:0<x<1\right\} 
\]
which is a fundamental domain for $\Gamma_{\infty}\eqdf\langle z\mapsto z+1\rangle$
and then define 
\[
F\eqdf\left\{ z\in F_{\infty}:\text{Im}z>\text{Im}\gamma z\text{ for all }\gamma\in\sigma^{-1}_{\mathfrak{a}}\Gamma\sigma_{\mathfrak{a}}\backslash\Gamma_{\infty}\right\} ,
\]
the region of $F_{\infty}$ exterior to the isometric circles $C_{g}$
with $\gamma\in\sigma^{-1}_{\mathfrak{a}}\Gamma\sigma_{\mathfrak{a}}\backslash\Gamma_{\infty}$.
By definition (\ref{eq:def-c_a}), $\frac{1}{c_{\mathfrak{a}}}$ is
precisely the radius of the largest isometric circle $C_{g}$ with
$\gamma\in\sigma^{-1}_{\mathfrak{a}}\Gamma\sigma_{\mathfrak{a}}\backslash\Gamma_{\infty}$
. Since $F$ is a fundamental domain for $\sigma^{-1}_{\mathfrak{a}}\Gamma\sigma_{\mathfrak{a}}$,
it contains the semi-strip $\left\{ x+iy\in\mathbb{H}\mid0<x<1,y>1\right\} $
so we have 
\[
\frac{1}{c_{\mathfrak{a}}}\leqslant1
\]
which implies that $\min\mathcal{C}_{\mathfrak{a}\mathfrak{b}}\geqslant1$.
The lemma follows from (\ref{eq:dirichlet-matrix}) and (\ref{eq:Dirichlet-determinant})
since we can write 
\[
\varphi(s)=\sum_{\sigma\in S_{n}}\text{sgn}(\sigma)\prod^{n}_{i=1}\varphi_{\mathfrak{a}_{i},\mathfrak{a}_{\sigma(i)}}(s).
\]
\end{proof}

\section{Eigenvalue comparison\label{sec:Eigenvalue-comparison}}

The main idea for the proof of Theorem \ref{thm:main-thm} will
be to exploit the existence of the many pants with fairly short boundary
guaranteed by Lemma \ref{thm:many-pants}. We aim to use the fact
that they contribute many closed geodesics to the geometric side of
the trace formula for the surface $X$. The way we obtain these geodesics
will be through Theorem \ref{thm:Naud-PGT}. However, we can only
use this result if the pants boundaries live in some fixed compact
interval away from zero. 

For the pants obtained in Lemma \ref{thm:many-pants}, this need not
be the case; their lengths could be arbitrarily close to zero. To
overcome this, we will show that the geometry and spectral behaviour
of the small eigenvalues of $X$ are close to those on the surface
obtained by collapsing down to nodes, all of the closed geodesics
shorter than some threshold. In doing so, this produces two cusps
in place of the geodesic. This will also collapse the boundaries of
the pants that we obtained to cusps if they were too short, and it
turns out that these pants are much more well behaved.

In this section, we will prove that this collapsing can be done in
a controlled manner, so it does not affect the geometry of the rest
of the surface too much (Proposition \ref{prop:bi-lips-map} and Lemma
\ref{lem:curves-on-limit-surface}). We will then prove that we can
pass small eigenvalues from the collapsed surface to the original
surface (Proposition \ref{prop:passing-evalues}). A version of the
idea to consider this collapsed surface appears in the work of Colbois
and Courtois \cite{Co.Co.1989}. 

\subsection{Collapsing short geodesics}

We first recall some basics about collars in hyperbolic surfaces. 

Around every simple closed geodesic $\gamma$ with length $\ell(\gamma)$
there is a neighbourhood that is isometric to a hyperbolic cylinder
\[
\mathcal{C}^{t}_{\gamma}=\left\{ (r,\theta):\ell(\gamma)\cosh(r)<t,0\leqslant\theta\leqslant2\pi\right\} ,
\]
with metric
\[
\mathrm{d}s^{2}=\mathrm{d}r^{2}+\left(\frac{\ell(\gamma)}{2\pi}\right)^{2}\cosh^{2}(r)\mathrm{d}\theta^{2}.
\]
By the collar theorem, $\mathcal{C}^{t}_{\gamma}$ is embedded for
any 
\[
t\leqslant\ell(\gamma)\cosh\left(\mathrm{arcsinh}\left(\sinh\left(\frac{\ell(\gamma)}{2}\right)^{-1}\right)\right)
\]
 but possibly larger, and for different $\gamma$ they are disjoint.
The collar $\mathcal{C}^{t}_{\gamma}$ has width $w(\ell)=\mathrm{arcosh}\left(\frac{t}{\ell(\gamma)}\right)$
and its boundary curves have length $t$.

Around every cusp $\mathfrak{c}$, there is a horocycle neighbourhood
isometric to a cusp domain $\mathcal{H}^{t}_{\mathfrak{c}}$ parametrised
as

\[
\mathcal{H}^{t}_{\mathfrak{c}}=\left\{ (r,\theta):r\geqslant0,0\leqslant\theta\leqslant2\pi\right\} ,
\]
with metric

\[
\mathrm{d}s^{2}=\mathrm{d}r^{2}+\left(\frac{t}{2\pi}\right)^{2}e^{-2r}\mathrm{d}\theta^{2}.
\]
In this parametrisation, the cusp is represented by the point at infinity
and the horocycle of length $t$ is at $r=0$. The neighbourhood embeds
for any $t\leqslant1$, but possibly larger, and are disjoint for
distinct cusps.

Let 
\[
\mathcal{H}_{\mathfrak{c}}(a,b)=\left\{ (r,\theta)\in\mathcal{H}^{1}_{\mathfrak{c}}:a\leqslant r\leqslant b,0\leqslant\theta\leqslant2\pi\right\} .
\]
 In particular, $\mathcal{H}_{\mathfrak{c}}(b,\infty)$ is the horocycle
neighbourhood around the cusp bounded by the horocycle of length $e^{-b}$.

Let $X^{(L)}_{\infty}$ denote the surface obtained from $X$ by pinching
all closed geodesics of length less than $L<1$ down to nodes, producing
two cusps in their place. The surface $X^{(L)}_{\infty}$ is possibly
disconnected. For $\gamma$ a closed geodesic on $X$ let 
\[
\beta_{\gamma}=\ell(\gamma)\cosh\left(\frac{1}{2}\mathrm{arcosh}\left(\frac{2}{\ell(\gamma)}\right)^{\frac{1}{4}}\right).
\]
It was established in \cite[Proposition 3.1]{Co.Co.1989} that there
exists a bi-Lipschitz homeomorphism
\begin{align*}
q & :\tilde{X}^{(L)}\eqdf X\setminus\bigcup_{\gamma}\mathcal{C}^{\beta_{\gamma}}_{\gamma}\to\tilde{X}^{(L)}_{\infty}\eqdf X^{(L)}_{\infty}\setminus\bigcup_{\gamma}\mathcal{H}_{\mathfrak{c}^{(1)}_{\gamma}}(-\log\beta_{\gamma},\infty)\cup\mathcal{H}_{\mathfrak{c}^{(2)}_{\gamma}}(-\log\beta_{\gamma},\infty)
\end{align*}
where the union runs over all closed geodesics of length less than
$L$ in the first surface and in the second surface, it runs over
the cusps $\mathfrak{c}^{(1)}_{\gamma}$ and $\mathfrak{c}^{(2)}_{\gamma}$
arising from pinching down the geodesics $\gamma$ to nodes, for $L$
sufficiently small. By bi-Lipschitz homeomorphism we mean a homeomorphism
satisfying for all $x,y$,
\[
\frac{1}{K[q]}d(x,y)\leq d(q(x),q(y))\leq K[q]d(x,y),
\]
for some $K[q]>0$ called the distortion factor (in \cite{Co.Co.1989}
or \cite[Definition 3.2.3]{Bu2010} this is called a quasi-isometry)
. 

In \cite[Proposition 3.1]{Co.Co.1989} it is also shown that $K[q]\to1$
as we choose $L\to0$. Since we pinch the geodesics down whose lengths
are shorter than some fixed threshold, we will want to prove that
$q$ can be constructed with an explicit upper bound on $K[q]$. To
this end, we will use a result of Buser, Makover, Muetzel and Silhol
\cite{Bu.Ma.Mu.Si2014} about constructing an explicit map in pairs
of pants to degenerate a boundary to a cusp with control over the
bi-Lipschitz constant.

Let $\gamma$ be a geodesic of length $0<\ell\leqslant2$ and define
$t_{\ell}=1+\frac{\ell^{2}}{4}$. Let $\mathcal{C}_{\ell}$ denote
the one-sided collar of width $w_{\ell}=\log\frac{2}{\ell}$ around
$\gamma$ which is non-empty when $\ell\leqslant2$ as this is when
$w_{\ell}$ is non-negative. It has boundary length $t_{\ell}$. For
a cusp, let $\mathcal{C}_{0}=\mathcal{H}^{1}$ be the horocycle of
length $1$. For $a,b,c\geqslant0$ denote by $Y_{a,b,c}$ the pair
of pants with boundary lengths $a,b,c$ with length zero being interpreted
as a cusp. 
\begin{lem}
\label{lem:buser-lemma}\cite[Theorem 5.1]{Bu.Ma.Mu.Si2014} Let $\ell\leqslant\frac{1}{2}$.
There is a quasiconformal homeomorphism 
\[
\phi:Y_{a,b,\ell}\to Y_{a,b,0}\setminus\mathcal{H}^{\frac{2}{\pi}\ell}
\]
where $\mathcal{H}^{t}$ is the horocycle neighborhood of the new
cusp bounded by the horocycle of length $t$. The map has the following
properties:
\begin{enumerate}
\item $\phi$ maps the boundaries of length $a$ and $b$ in $Y_{a,b,\ell}$
to the boundaries of length $a$ and $b$ in $Y_{a,b,0}$ and it maps
the boundary of length $\ell$ to the horocycle of length $\frac{2}{\pi}\ell$.
$\phi$ also preserves the boundary constant speed parametrisation
and has a quasiconformal dilation of at most $1+2\ell^{2}$.
\item $\phi$ restricted to $Y_{a,b,\ell}\setminus(\mathcal{C}_{a}\sqcup\mathcal{C}_{b}\sqcup\mathcal{C}_{\ell})$
is a homeomorphism with image $Y_{a,b,0}\setminus(\mathcal{C}_{a}\sqcup\mathcal{C}_{b}\sqcup\mathcal{H}^{1})$
where we recall that $\mathcal{C}_{a}$ and $\mathcal{C}_{b}$ can
be empty if $a,b\geqslant2$ and they are the horocycles of length
$1$ if $a,b=0$.
\item The restricted map to $Y_{a,b,\ell}\setminus(\mathcal{C}_{a}\sqcup\mathcal{C}_{b}\sqcup\mathcal{C}_{\ell})$
maps the boundary of length $t_{\ell}$ to the horocycle of length
$1$ and the boundaries of lengths $t_{a}$ and $t_{b}$ (or just
$a$ and $b$ in case they have length $\geqslant2$) to their corresponding
identical boundaries. Moreover, it preserves the boundary constant
speed parametrisations. The restriction is bi-Lipschitz with distortion
factor bounded by $1+\frac{5}{2}\ell^{2}$.
\item The restriction of $\phi$ to $\mathcal{C}_{a}$ and $\mathcal{C}_{b}$
is an isometry onto their corresponding images.
\end{enumerate}
\end{lem}

The boundary preservation property in points 1 and 3 is called boundary
coherence of the map $\phi$ by the authors therein.

Notice that for $\gamma$ of length $\ell<L<\frac{1}{2}$, we have
that $\beta_{\gamma}<1$ is shorter than $t_{\ell}$. For our homeomorphism
$q$ we will thus wish to cut off a smaller collar than $\mathcal{C}_{\ell}$
around the geodesic $\gamma$ when it has length $<L$. Moreover,
in the image we will only want to cut off the horocycle $\mathcal{H}(-\log\beta_{\gamma},\infty)$
rather than $\mathcal{H}(0,\infty)$. 

The restriction of $\phi$ to $Y_{a,b,\ell}\setminus\mathcal{C}_{\ell}$
can be used to build this. Indeed, we will glue it together with a
nicely controlled map that sends the annulus $\mathcal{C}_{\ell}\setminus\mathcal{C}^{\beta_{\gamma}}_{\gamma}$
to the annulus $\mathcal{\mathcal{H}}(0,-\log\beta_{\gamma})$ . This
will then provide a bi-Lipschitz map between the $Y_{a,b,\ell}\setminus\mathcal{C}^{\beta_{\gamma}}_{\gamma}$
and $Y_{a,b,0}\setminus\mathcal{H}(-\log\beta_{\gamma},\infty)$. 

It is possible that in $X$ there may be pants with multiple boundaries
that are shorter than $L$ and so they will also need to be collapsed.
In this case, they will be short enough so that the $\mathcal{C}_{a}$
or $\mathcal{C}_{b}$ are non-empty and the boundary coherence property
of the restriction in 2 along with the isometric action in 4 will
allow us to compose the maps to collapse multiple boundaries successively.
This idea was also carried out in \cite[Corollary 5.2]{Bu.Ma.Mu.Si2014}.

The idea then will be to pick a pants decomposition of $X$ that includes
all of the geodesics of length at most $L$, collapse each of the
boundaries in the pants appropriately and glue the maps together.
Since the maps will act on different pants, the resulting bi-Lipschitz
distortions will be the maximum of the individual maps. We begin with
proving the existence of a bi-Lipschitz homeomorphism that maps the
annuli as above.
\begin{lem}
\label{lem:annulus-maps}Let $\gamma$ be a closed geodesic with length
$\ell<0.1$. With the notation above, there is an orientation-preserving
bi-Lipschitz homeomorphism 

\[
\Upsilon:\mathcal{C}_{\ell}\setminus\mathcal{C}^{\beta_{\gamma}}_{\gamma}\to\mathcal{H}(0,-\log\beta_{\gamma})
\]
with bi-Lipschitz distortion factor at most $1+\exp\left(-\mathrm{arcosh}\left(\frac{2}{\ell}\right)^{\frac{1}{4}}\right)$.
It maps the boundary of length $\beta_{\gamma}$ to the horocycle
of length $\beta_{\gamma}$ and the boundary of length $t_{\ell}$
to the horocycle of length $1$. $\Upsilon$ also preserves the constant
speed parameterisations of the boundaries and horocycles.
\end{lem}

\begin{proof}
For brevity, let $b_{\ell}=\frac{1}{2}\mathrm{arcosh}\left(\frac{2}{\ell}\right)^{\frac{1}{4}}$
be the width of the collar $\mathcal{C}^{\beta_{\gamma}}_{\gamma}$
and recall $w_{\ell}=\log\left(\frac{2}{\ell}\right)$. On $\mathcal{C}_{\ell}\setminus\mathcal{C}^{\beta_{\gamma}}_{\gamma}$
we have the coordinates

\[
\left\{ (r,\theta):b_{\ell}\leqslant r\leqslant w_{\ell},0\leqslant\theta\leqslant2\pi\right\} 
\]
with metric
\[
\mathrm{d}s^{2}=\mathrm{d}r^{2}+\left(\frac{\ell}{2\pi}\right)^{2}\cosh^{2}(r)\mathrm{d}\theta^{2}.
\]
On $\mathcal{H}(0,-\log\beta_{\gamma})$ we have the coordinates 
\[
\left\{ (r,\theta):0\leqslant r\leqslant-\log\beta_{\gamma},0\leqslant\theta\leqslant2\pi\right\} 
\]
with metric 
\[
\mathrm{d}s^{2}=\mathrm{d}r^{2}+\left(\frac{t}{2\pi}\right)^{2}e^{-2r}\mathrm{d}\theta^{2}.
\]
We describe the map explicitly (note it is essentially the map considered
in \cite{Co.Co.1989}):

\[
\Upsilon:(r,\theta)\mapsto(c_{\ell}(w_{\ell}-r),-\theta),
\]
where $c_{\ell}=\frac{-\log\beta_{\gamma}}{w_{\ell}-b_{\gamma}}$.
The minus sign on the $\theta$ is to make the map orientation preserving. 

Since the horocycles of length $0$ and $\beta_{\gamma}$ are represented
by $r=0$ and $r=-\log\beta_{\gamma}$ respectively, the boundaries
are mapped to the horocycles correctly. Moreover, the boundaries are
mapped to preserve the constant speed parametrisations. Also $\Upsilon$
is clearly a homeomorphism, and so it remains to show it is bi-Lipschitz
with the stated distortion. We compute this by finding a $K\geq1$
such that 
\[
\frac{1}{K}|v|\leqslant|d\Upsilon(v)|\leqslant K|v|,
\]
for any tangent vector $v$. 

To this end, notice that $\left\{ \partial_{r},\frac{2\pi}{\ell\cosh(r)}\partial_{\theta}\right\} $
and $\left\{ \partial_{r},2\pi e^{r}\partial_{\theta}\right\} $ are
orthonormal bases of tangent vectors for the two respective domains.
We can compute
\begin{align*}
d\Upsilon(\partial_{r}) & =-c_{\ell}\partial_{r},\\
d\Upsilon\left(\frac{2\pi}{\ell\cosh(r)}\partial_{\theta}\right) & =-\frac{e^{-c_{\ell}(w_{\ell}-r)}}{\ell\cosh(r)}\left(2\pi e^{c_{\ell}(w_{\ell}-r)}\partial_{\theta}\right).
\end{align*}
This means that if $v=x\partial_{r}+y\frac{2\pi}{\ell\cosh(r)}\partial_{\theta}$
is a tangent vector in the orthonormal basis so that $|v|^{2}=x^{2}+y^{2}$,
then
\[
|d\Upsilon(v)|^{2}=c^{2}_{\ell}x^{2}+\left(\frac{e^{-c_{\ell}(w_{\ell}-r)}}{\ell\cosh(r)}\right)^{2}y^{2}.
\]
Set $\sigma(\ell,r)=\frac{e^{-c_{\ell}(w_{\ell}-r)}}{\ell\cosh(r)}$.
Then,
\[
\min\left\{ c_{\ell},\sigma(\ell,r)\right\} |v|\leq|d\Upsilon(v)|\leq\max\left\{ c_{\ell},\sigma(\ell,r)\right\} |v|.
\]
The bi-Lipschitz constant $K$ is hence bounded above by 
\begin{equation}
K\leqslant\max\left\{ c_{\ell},\sigma(\ell,r),\frac{1}{c_{\ell}},\frac{1}{\sigma(\ell,r)}\right\} .\label{eq:bi-lips-bound}
\end{equation}
First, notice that 
\begin{align}
c_{\ell} & =\frac{-\log(\ell\cosh(b_{\ell}))}{w_{\ell}-b_{\ell}}\nonumber \\
 & =\frac{w_{\ell}-\log\left(e^{b_{\ell}}+e^{-b_{\ell}}\right)}{w_{\ell}-b_{\ell}}\nonumber \\
 & =1-\frac{\log(1+e^{-2b_{\ell}})}{w_{\ell}-b_{\ell}}.\label{eq:c_ell}
\end{align}
It is clear then that $c_{\ell}\leqslant1$. As for a lower bound,
we use the fact that $\log(1+x)\leqslant x$ so that

\[
c_{\ell}\geqslant1-\frac{e^{-2b_{\ell}}}{w_{\ell}-b_{\ell}}.
\]
Since $\mathrm{arcosh}(x)\leqslant\log(2x)$, we can readily check
that 
\[
w_{\ell}-b_{\ell}\geqslant\log\left(\frac{2}{\ell}\right)-\frac{1}{2}\log\left(\frac{4}{\ell}\right)^{\frac{1}{4}}\geqslant2
\]
whenever $\ell<0.1$. But, since $b_{\ell}>0$, we have $1+e^{-2b_{\ell}}\leqslant2$,
and so 
\[
c_{\ell}\geqslant\frac{1}{1+e^{-2b_{\ell}}}.
\]
Next, using the expression (\ref{eq:c_ell}) for $c_{\ell}$ and $w_{\ell}=\log\left(\frac{2}{\ell}\right)$
we have
\begin{align*}
\sigma(\ell,r) & =\frac{1}{\ell\cosh(r)}\exp\left(-(w_{\ell}-r)+\log\left(1+e^{-2b_{\ell}}\right)\frac{w_{\ell}-r}{w_{\ell}-b_{\ell}}\right)\\
 & =\frac{e^{r}}{2\cosh(r)}\left(1+e^{-2b_{\ell}}\right)^{\frac{w_{\ell}-r}{w_{\ell}-b_{\ell}}}\\
 & =\frac{1}{1+e^{-2r}}\left(1+e^{-2b_{\ell}}\right)^{\frac{w_{\ell}-r}{w_{\ell}-b_{\ell}}}.
\end{align*}
Since $r\geqslant b_{\ell}$, we have $\frac{w_{\ell}-r}{w_{\ell}-b_{\ell}}\leqslant1$,
and so immediately get the upper bound 
\[
\sigma(\ell,r)\leqslant1+e^{-2b_{\ell}}.
\]
Using instead that $\left(1+e^{-2b_{\ell}}\right)^{\frac{w_{\ell}-r}{w_{\ell}-b_{\ell}}}\geqslant1$
and $r\geqslant b_{\ell}$, we obtain the lower bound 
\[
\sigma(\ell,r)\geqslant\frac{1}{1+e^{-2b_{\ell}}}.
\]
Combining all of these bounds into (\ref{eq:bi-lips-bound}) then
gives the result.
\end{proof}

We now build the desired bi-Lipschitz homeomorphism $q$.
\begin{prop}
\label{prop:bi-lips-map}Suppose that $L<0.1$. There exists a bi-Lipschitz
homeomorphism 
\begin{align*}
q & :\tilde{X}^{(L)}=X\setminus\bigcup_{\gamma}\mathcal{C}^{\beta_{\gamma}}_{\gamma}\to\tilde{X}^{(L)}_{\infty}=X^{(L)}_{\infty}\setminus\bigcup_{\gamma}\mathcal{H}_{\mathfrak{c}^{(1)}_{\gamma}}(-\log\beta_{\gamma},\infty)\cup\mathcal{H}_{\mathfrak{c}^{(2)}_{\gamma}}(-\log\beta_{\gamma},\infty),
\end{align*}
where the union runs over all closed geodesics of length less than
$L$ in the first surface and in the second surface, it runs over
the cusps $\mathfrak{c}^{(1)}_{\gamma}$ and $\mathfrak{c}^{(2)}_{\gamma}$
arising from pinching down the geodesics $\gamma$ to nodes. Moreover,
the distortion factor satisfies $K[q]\leqslant\max\left\{ \left(1+\frac{5}{2}L^{2}\right)^{3},1+\exp\left(-\mathrm{arcosh}\left(\frac{2}{L}\right)^{\frac{1}{4}}\right)\right\} $.
\end{prop}

\begin{proof}
Consider the collection of all closed geodesics of length at most
$L$ and complete them to a pants decomposition of $X$. On any such
pants, it could be that any collection of its boundaries have length
shorter than $L$. 

Suppose first that we have a pair of pants where there is only one
boundary $\gamma$ of length $\ell\leqslant L$ that gets collapsed,
and let the other two boundaries have lengths $a$ and $b$. It is
isometric to the pants $Y_{a,b,\ell}$ and we want the map $q$ when
restricted to $Y_{a,b,\ell}\setminus\mathcal{C}^{\beta_{\gamma}}_{\gamma}$
to map into the corresponding chopped pants $Y_{a,b,0}\setminus\mathcal{H}_{\mathfrak{c}_{\gamma}}(-\log\beta_{\gamma},\infty)$
contained in $X^{(L)}_{\infty}$, where we suppose that $\mathfrak{c}_{\gamma}$
is the corresponding cusp.

To this end, we take the bi-Lipschitz homeomorphism $\phi|_{Y_{a,b,\ell}\setminus\mathcal{C}_{\ell}}$
from Lemma \ref{lem:buser-lemma} which has image $Y_{a,b,0}\setminus\mathcal{H}_{\mathfrak{c}_{\gamma}}(0,\infty)$
and distortion bounded by $1+\frac{5}{2}L^{2}$ (because it acts isometrically
on $\mathcal{C}_{a}$ and $\mathcal{C}_{b}$ in the chance they are
non-empty). It also has the boundary coherence property, mapping boundaries
to corresponding boundaries/horocycles in a way that preserves their
constant speed parametrisation.

We now take the map $\Upsilon$ from Lemma \ref{lem:annulus-maps}.
Because it maps the boundaries of length $\beta_{\gamma}$ and $t_{\ell}$
to the respective horocycles of lengths $\beta_{\gamma}$ and $1$,
and preserves their unit speed parametrisations, it matches perfectly
with $\phi|_{Y_{a,b,\ell}\setminus\mathcal{C}_{\ell}}$ along the
common boundary. The distortion factor of $\Upsilon$ is bounded by
\[
1+e^{-\frac{1}{2}\mathrm{arcosh}\left(\frac{2}{\ell}\right)^{\frac{1}{4}}}\leqslant1+e^{-\frac{1}{2}\mathrm{arcosh}\left(\frac{2}{L}\right)^{\frac{1}{4}}}.
\]
We now glue $\phi|_{Y_{a,b,\ell}\setminus\mathcal{C}_{\ell}}$ and
$\Upsilon$ together. In doing so, we obtain a map from $Y_{a,b,\ell}\setminus\mathcal{C}^{\beta_{\gamma}}_{\gamma}$
to $Y_{a,b,0}\setminus\mathcal{H}_{\mathfrak{c}_{\gamma}}(-\log\beta_{\gamma},\infty)$
such that the boundaries of lengths $a$ and $b$ are mapped to their
corresponding boundaries along with their constant speed parametrisations.
Moreover, since the map is obtained from gluing, the bi-Lipschitz
constant is bounded by $\max\left\{ 1+\frac{5}{2}L^{2},1+e^{-\frac{1}{2}\mathrm{arcosh}\left(\frac{2}{L}\right)^{\frac{1}{4}}}\right\} $.

The case where we want to collapse two or three boundaries is similar
and involves composing maps from Lemma \ref{lem:buser-lemma}. We
describe the procedure for two boundaries since it is the same for
three. Let $\gamma_{1}$ and $\gamma_{2}$ be the geodesics with lengths
$\ell_{1},\ell_{2}\leqslant L$ and let $\mathfrak{c}_{1}$ and $\mathfrak{c}_{2}$
denote the corresponding cusps. By Lemma \ref{lem:buser-lemma} we
obtain two maps $\phi_{1}:Y_{a,\ell_{1},\ell_{2}}\to Y_{a,0,\ell_{2}}\setminus\mathcal{H}^{\frac{2}{\pi}\ell_{1}}_{\mathfrak{c}_{1}}$
and $\phi_{2}:Y_{a,0,\ell_{2}}\to Y_{a,0,0}\setminus\mathcal{H}^{\frac{2}{\pi}\ell_{2}}_{\mathfrak{c}_{2}}$
such that 
\[
\phi_{1}|_{Y_{a,\ell_{1},\ell_{2}}\setminus\left(\mathcal{C}_{\ell_{1}}\sqcup\mathcal{C}_{\ell_{2}}\right)}:Y_{a,\ell_{1},\ell_{2}}\setminus\left(\mathcal{C}_{\ell_{1}}\sqcup\mathcal{C}_{\ell_{2}}\right)\to Y_{a,0,\ell_{2}}\setminus\left(\mathcal{H}^{1}_{\mathfrak{c}_{1}}\sqcup\mathcal{C}_{\ell_{2}}\right)
\]
is a bi-Lipschitz homeomorphism with distortion bounded by $1+\frac{5}{2}\ell^{2}_{1}\leqslant1+\frac{5}{2}L^{2}$
and 
\[
\phi_{2}|_{Y_{a,0,\ell_{2}}\setminus\left(\mathcal{H}^{1}_{\mathfrak{c}_{1}}\sqcup\mathcal{C}_{\ell_{2}}\right)}:Y_{a,0,\ell_{2}}\setminus\left(\mathcal{H}^{1}_{\mathfrak{c}_{1}}\sqcup\mathcal{C}_{\ell_{2}}\right)\to Y_{a,0,0}\setminus\left(\mathcal{H}^{1}_{\mathfrak{c}_{1}}\sqcup\mathcal{H}^{1}_{\mathfrak{c}_{2}}\right)
\]
is a bi-Lipschitz homeomorphism with distortion bounded by $1+\frac{5}{2}\ell^{2}_{2}\leqslant1+\frac{5}{2}L^{2}$.
Moreover, at each stage, the corresponding boundaries and horocycles
are mapped to one another such that their constant speed parametrisations
are preserved. 

We can now compose these restricted maps to obtain a bi-Lipschitz
homeomorphism $\phi:Y_{a,\ell_{1},\ell_{2}}\setminus\left(\mathcal{C}_{\ell_{1}}\sqcup\mathcal{C}_{\ell_{2}}\right)\to Y_{a,0,0}\setminus\left(\mathcal{H}^{1}_{\mathfrak{c}_{1}}\sqcup\mathcal{H}^{1}_{\mathfrak{c}_{2}}\right)$
with distortion bounded by $\left(1+\frac{5}{2}L^{2}\right)^{2}$.
In the case of three cusps, we would obtain a third map in the composition
and the distortion factor would have exponent 3.

For the remaining annuli, we use Lemma \ref{lem:annulus-maps} to
obtain bi-Lipschitz homeomorphisms with distortion bounded by $1+e^{-\frac{1}{2}\mathrm{arcosh}\left(\frac{2}{L}\right)^{\frac{1}{4}}}$.
We then glue the two annuli maps with $\phi$ to obtain the desired
homeomorphism that sends $Y_{a,\ell_{1},\ell_{2}}\setminus\left(\mathcal{C}^{\beta_{\gamma_{1}}}_{\gamma_{1}}\sqcup\mathcal{C}^{\beta_{\gamma_{2}}}_{\gamma_{2}}\right)$
to $Y_{a,0,0}\setminus\left(\mathcal{H}_{\mathfrak{c}_{1}}(-\log\beta_{\gamma_{1}},\infty)\sqcup\mathcal{H}_{\mathfrak{c}_{2}}(-\log\beta_{\gamma_{2}},\infty)\right)$
and has bi-Lipschitz distortion bounded by $\max\left\{ \left(1+\frac{5}{2}L^{2}\right)^{2},1+e^{-\frac{1}{2}\mathrm{arcosh}\left(\frac{2}{L}\right)^{\frac{1}{4}}}\right\} $.
For three cusps, the exponent of the first term becomes $3$.

The map $q$ is now obtained by gluing together each of the individual
pants maps; on a pants where no boundaries are collapsed we do the
identity map. The gluing is possible because each of our pants maps
ensured that the boundaries that are not collapsed are mapped to themselves
with the correct constant speed parametrisation. Since we are gluing
bi-Lipschitz homeomorphisms, the result is a bi-Lipschitz homeomorphism
with distortion factor bounded by the maximum among all of the glued
maps which is bounded uniformly by $\max\left\{ \left(1+\frac{5}{2}L^{2}\right)^{3},1+e^{-\frac{1}{2}\mathrm{arcosh}\left(\frac{2}{L}\right)^{\frac{1}{4}}}\right\} $.
\end{proof}

We now prove a consequence of the existence of the homeomorphism $q$
that allows us to compare the lengths of closed geodesics between
$X$ and $X_{\infty}$.
\begin{lem}
\label{lem:curves-on-limit-surface}Let $X$ be a hyperbolic surface
with signature $(g,n)$ and let $L<0.1$. Let $X^{(L)}_{\infty}$
be the (possibly disconnected) surface obtained above by shrinking
the simple closed geodesics of length $<L$ down to nodes, and suppose
that it has genus $g_{\infty}$ and $n_{\infty}$ cusps. Then, $\mathrm{sys}(X_{\infty})>\frac{1}{2}L$.
\end{lem}

\begin{proof}
If $\mathrm{sys}(X)>L$, then $X=X^{(L)}_{\infty}$ and there is nothing
to prove, so suppose this is not the case. Let $\gamma$ be a closed
geodesic on $X^{(L)}_{\infty}$ so that $\gamma$ is not contractible
or freely homotopic to a cusp. Moreover, assume that it has length
at most $\frac{1}{2}$; then $\gamma$ must lie entirely in $\tilde{X}^{(L)}_{\infty}$.
Indeed, the distance between the horocycles of length $\beta_{\alpha}$
and $1$ around the cusps in $X^{(L)}_{\infty}$ that arise from collapsing
geodesics $\alpha$ in $X$ is $-\log\beta_{\alpha}>1$ since $L<0.1$.
So, if $\gamma$ entered the horocycle neighbourhood bounded by the
horocycle of length $\beta_{\alpha}$, then it must lie fully within
the length 1 horocycle neighbourhood of one of the cusps which are
embedded cuspidal domains. Thus, $\gamma$ must be contractible or
homotopic to the cusp which is not the case.

Consider $q^{-1}(\gamma)$, which is now a closed curve in $\tilde{X}^{(L)}$.
Since $q^{-1}$ is a homeomorphism, it preserves the free homotopy
classes of curves and, in particular, it can only send contractible
curves to contractible curves and curves freely homotopic to cusps
or boundaries to curves freely homotopic to cusps or boundaries. Since
$\gamma$ falls into none of these cases (it cannot be homotopic to
a boundary in $\tilde{X}^{(L)}_{\infty}$ because these are horocycles,
and hence it would be homotopic to a cusp in $X^{(L)}_{\infty}$),
$q^{-1}(\gamma)$ is homotopic to a closed geodesic in $\tilde{X}^{(L)}$.
The preservation of free homotopy classes under $q^{-1}$ also means
that this geodesic is unique, that is, non-homotopic closed geodesics
on $\tilde{X}^{(L)}_{\infty}$ map to non-homotopic curves on $\tilde{X}^{(L)}$.

Since all of the closed geodesics on $\tilde{X}^{(L)}$ have length
$>L$, and $q^{-1}$ is bi-Lipschitz with distortion factor $K[q]\leqslant2$
it follows that $\ell(\gamma)\geqslant\frac{1}{K[q]}\ell(q^{-1}(\gamma))>\frac{1}{2}L$
as desired. 
\end{proof}

\subsection{Comparing eigenvalues \label{subsec:Comparing-eigenvalues}}

In the previous section, we showed how we can compare the geometry
and geodesics between $X$ and $X^{(L)}_{\infty}$. In this section,
we will show we can compare their small eigenvalues. We will make
use of the min-max theorem which can be found for example in \cite{Sm12}.
\begin{lem}[Min-max]
\label{lem:Minmax}Let $A$ be a non-negative self-adjoint operator
on a Hilbert space $H$ with domain $\mathcal{D}\left(A\right)$.
Let $\lambda_{1}\leqslant\dots\leqslant\lambda_{k}$ denote the eigenvalues
of $A$ below the essential spectrum $\sigma_{ess}\left(A\right)$.
Then for $1\leqslant j\leqslant k,$
\[
\lambda_{j}=\min_{\psi_{1},\dots,\psi_{j}}\max\left\{ \frac{\langle\psi,A\psi\rangle}{\|\psi\|^{2}}:0\neq\psi\in\textup{span}\left(\psi_{1},\dots,\psi_{j}\right)\right\} ,
\]
where the minimum is taken over linearly independent $\psi_{1},\dots,\psi_{j}\in\mathcal{D}\left(A\right)$.
If $A$ only has $l\geqslant0$ eigenvalues below the essential spectrum
then for any integer $s\geqslant1$,
\[
\inf\sigma_{ess}\left(A\right)=\inf_{\psi_{1},\dots,\psi_{l+s}}\sup\left\{ \frac{\langle\psi,A\psi\rangle}{\|\psi\|^{2}}:0\neq\psi\in\textup{span}\left(\psi_{1},\dots,\psi_{l+s}\right)\right\} ,
\]
where again the infimum is taken over linearly independent $\psi_{1},\dots,\psi_{\ell+s}\in\mathcal{D}\left(A\right)$.
\end{lem}

The following proposition and its proof were suggested to us by an
anonymous referee, we are very grateful for this suggestion. Similar
ideas appear in work of Li \cite{Li20} and Ballman-Polymerakis \cite{Ba.Po.24}.
\begin{prop}
\label{prop:passing-evalues}Let $0<d<\frac{1}{4}$ be given and suppose
that $\eta>0$ is chosen to satisfy
\[
\eta^{4}\left(\frac{1}{4}-\frac{d}{2}\right)\leqslant\frac{1}{4}-\frac{d}{4}.
\]
Then, for any $0<L<0.1$ that satisfies
\begin{equation}
\max\left\{ \left(1+\frac{5}{2}L^{2}\right)^{3},1+\exp\left(-\mathrm{arcosh}\left(\frac{2}{L}\right)^{\frac{1}{4}}\right)\right\} <\eta\label{eq:bilips-UB}
\end{equation}
and 
\begin{equation}
-\log\left(L\cosh\left(\frac{1}{2}\mathrm{arcosh}\left(\frac{2}{L}\right)^{\frac{1}{4}}\right)\right)\geqslant\sqrt{\frac{8}{d}},\label{eq:dist-LB}
\end{equation}
 if $X^{(L)}_{\infty}$ has at least $m$ eigenvalues smaller than
$\frac{1}{4}-d$, then $X$ has at least $m$ eigenvalues smaller
than $\frac{1}{4}-\frac{d}{4}$.
\end{prop}

\begin{rem}
Given any $0<d<\frac{1}{4}$, such an $L$ always exists. Indeed,
the condition on $\eta$ allows us to choose $\eta>1+\varepsilon$
for some $\varepsilon>0$ depending only upon $d$. Since both $\frac{5}{2}L^{2}$
and $\exp\left(-\mathrm{arcosh}\left(\frac{2}{L}\right)^{\frac{1}{4}}\right)$
tend to zero as $L\to0$, we can ensure that $L$ is small, so that
their maximum is less than $\varepsilon$. Now, $-\log\left(L\cosh\left(\frac{1}{2}\mathrm{arcosh}\left(\frac{2}{L}\right)^{\frac{1}{4}}\right)\right)\to\infty$
as $L\to0$, and so by shrinking $L$ further, we can ensure that
(\ref{eq:dist-LB}) is also satisfied.
\end{rem}

\begin{proof}
In $\tilde{X}^{(L)}_{\infty}$ we cut off the horocycle neighbourhoods
that are bounded by the horocycles of lengths at most
\[
L\cosh\left(\frac{1}{2}\mathrm{arcosh}\left(\frac{2}{L}\right)^{\frac{1}{4}}\right)<1.
\]
The distance between these horocycles and the length 1 horocycle of
the same cusp is then
\[
h_{\max}:=-\log\left(L\cosh\left(\frac{1}{2}\mathrm{arcosh}\left(\frac{2}{L}\right)^{\frac{1}{4}}\right)\right)\geqslant\sqrt{\frac{8}{d}}.
\]
Now, let $\left\{ \varphi_{1},\ldots,\varphi_{m}\right\} $ be an
orthonormal basis of eigenfunctions for the first $m$ eigenvalues
of $X^{(L)}_{\infty}$. Choose a function $\chi\in C^{\infty}(X^{(L)}_{\infty})$
that is zero on and in a smaller neighbourhood of 

\[
V_{L}=\bigcup_{\gamma}\mathcal{H}_{\mathfrak{c}^{(1)}_{\gamma}}(-\log\beta_{\gamma},\infty)\cup\mathcal{H}_{\mathfrak{c}^{(2)}_{\gamma}}(-\log\beta_{\gamma},\infty),
\]
where the union runs over all of the closed geodesics of length $\leqslant L$
in $X$, and is equal to $\frac{\pi}{2}$ on and in a neighbourhood
of
\[
X^{(L)}_{\infty}\setminus\bigcup_{\gamma}\mathcal{H}_{\mathfrak{c}^{(1)}_{\gamma}}(0,\infty)\cup\mathcal{H}_{\mathfrak{c}^{(2)}_{\gamma}}(0,\infty).
\]
Due to the lower bound on the distance between the two horocycles
over which $\chi$ smooths from $\frac{\pi}{2}$ to $0$, we can also
ensure than $\|\nabla\chi\|<\frac{2}{h_{\max}}$.

Now set $\chi_{1}=\sin\chi$ and $\chi_{2}=\cos\chi$ so that $\chi^{2}_{1}+\chi^{2}_{2}=1$.
If $f\in\mathrm{span}\left\{ \varphi_{1},\ldots,\varphi_{m}\right\} $
and $\|f\|_{2}=1$, then 
\[
\frac{1}{4}-d\geqslant\int_{X^{(L)}_{\infty}}\|\nabla f\|^{2}\mathrm{d}\mu_{X^{(L)}_{\infty}}=\sum_{i=1,2}\int_{X^{(L)}_{\infty}}\|\nabla(\chi_{i}f)\|^{2}\mathrm{d}\mu_{X^{(L)}_{\infty}}-\int_{X^{(L)}_{\infty}}\|\nabla\chi_{i}\|^{2}f^{2}\mathrm{d}\mu_{X^{(L)}_{\infty}}.
\]
The equality comes from the identity

\begin{align*}
\|\nabla(\chi_{i}f)\|^{2} & =\|\chi_{i}\nabla f+f\nabla\chi_{i}\|^{2}\\
 & =\chi^{2}_{i}\|\nabla f\|^{2}+f^{2}\|\nabla\chi_{i}\|^{2}+2f\langle\chi_{i}\nabla\chi_{i},\nabla f\rangle,
\end{align*}
and so upon summing and using $\chi^{2}_{1}+\chi^{2}_{2}=1$ we obtain
\[
\sum_{i=1,2}\|\nabla(\chi_{i}f)\|^{2}=\|\nabla f\|^{2}+\sum_{i=1,2}f^{2}\|\nabla\chi_{i}\|^{2}+2f\langle\chi_{1}\nabla\chi_{1}+\chi_{2}\nabla\chi_{2},\nabla f\rangle.
\]
Since $\nabla\chi^{2}_{i}=2\chi_{i}\nabla\chi_{i}$ we have

\[
0=\nabla(\chi^{2}_{1}+\chi^{2}_{2})=2(\chi_{1}\nabla\chi_{1}+\chi_{2}\nabla\chi_{2})
\]
and so the claimed equality holds. Now since 
\[
\|\nabla\chi_{1}\|^{2}+\|\nabla\chi_{2}\|^{2}=\cos^{2}(\chi)\|\nabla\chi\|^{2}+\sin^{2}(\chi)\|\nabla\chi\|^{2}=\|\nabla\chi\|^{2}<\frac{4}{h^{2}_{max}}\leqslant\frac{d}{2}
\]
 and $\|f\|_{2}=1$, we obtain 
\begin{equation}
\frac{1}{4}-\frac{d}{2}\geqslant\sum_{i=1,2}\int_{X^{(L)}_{\infty}}\|\nabla(\chi_{i}f)\|^{2}\mathrm{d}\mu_{X^{(L)}_{\infty}}.\label{eq:nabla-bd}
\end{equation}
Note that since $\cos(\frac{\pi}{2})=0$ and $\cos(0)=1$, the definition
of $\chi$ means that $\chi_{2}$ is supported entirely in the horocycle
neighbourhood domains bounded by horocycles of length $1$ and vanishes
near their boundary. Since these domains are isometric to cuspidal
domains, we have 
\begin{equation}
\int_{X^{(L)}_{\infty}}\|\nabla(\chi_{2}f)\|^{2}\mathrm{d}\mu_{X^{(L)}_{\infty}}\geqslant\frac{1}{4}\int_{X_{\infty}}(\chi_{2}f)^{2}\mathrm{d}\mu_{X^{(L)}_{\infty}}.\label{eq:cuspidal-lb}
\end{equation}
If $\chi_{1}f=0$, then since $\chi^{2}_{2}f^{2}=\chi^{2}_{1}f^{2}+\chi^{2}_{2}f^{2}=f^{2}$
and $\|f\|_{2}=1$, we would obtain a contradiction with (\ref{eq:nabla-bd})
and so $\chi_{1}f\neq0$. 

Since $\chi_{1}=\sin\chi$ is zero in a neighbourhood of $V_{L}$,
we see that $\chi_{1}$ acts by multiplication on $\mathrm{span}\left\{ \varphi_{1},\ldots,\varphi_{m}\right\} $
to send it to an $m$-dimensional subspace of domain of the Laplacian
on $\tilde{X}^{(L)}_{\infty}$. By pre-composing the functions in
this subspace with the bi-Lipschitz map $q$ from Proposition \ref{prop:bi-lips-map},
we obtain an $m$-dimensional subspace of the Laplacian on $\tilde{X}^{(L)}$,
and all of the functions smoothly vanish near the boundary by definition
of $\chi_{1}$ so we can extend them to zero on $X\setminus\tilde{X}^{(L)}$. 

In summary, we have obtained an $m$-dimensional subspace of the domain
of the Laplacian on $X$ given by $\mathrm{span}\left\{ \widetilde{(\chi_{1}\varphi_{1})\circ q},\ldots,\widetilde{(\chi_{1}\varphi_{m})\circ q}\right\} $,
where the tilde denotes the extension by zero. Note that (for example
by \cite[Proof of Theorem 14.9.2]{Bu2010}) for any sufficiently smooth
$f$ that vanishes on the boundaries we have

\[
\mathcal{R}(f\circ q,\tilde{X}^{(L)})\leqslant K[q]^{4}\mathcal{R}(f,\tilde{X}^{(L)}_{\infty}),
\]
where $\mathcal{R}(\cdot,\cdot)$ is the Rayleigh quotient.

We obtain for any $F\in\mathrm{span}\left\{ \widetilde{(\chi_{1}\varphi_{1})\circ q},\ldots,\widetilde{(\chi_{1}\varphi_{m})\circ q}\right\} $,
so that $F=\sum^{m}_{i=1}a_{i}\widetilde{(\chi_{1}\varphi_{i})\circ q}$,
that
\begin{align*}
\mathcal{R}(F,X) & =\mathcal{R}\left(\sum^{m}_{i=1}a_{i}(\chi_{1}\varphi_{i})\circ q,\tilde{X}^{(L)}\right)\\
 & \leqslant K[q]^{4}\mathcal{R}\left(\sum^{m}_{i=1}a_{i}\chi_{1}\varphi_{i},\tilde{X}^{(L)}_{\infty}\right)\\
 & \leqslant\eta^{4}\mathcal{R}\left(\sum^{m}_{i=1}a_{i}\chi_{1}\varphi_{i},\tilde{X}^{(L)}_{\infty}\right).
\end{align*}
Let $f=\sum^{m}_{i=1}a_{i}\varphi_{i}$. From (\ref{eq:nabla-bd})
and (\ref{eq:cuspidal-lb}), we have (using again that $\chi^{2}_{1}+\chi^{2}_{2}=1$)
that
\begin{align*}
\frac{\int_{\tilde{X}^{(L)}_{\infty}}\left\Vert \nabla(\chi_{1}f)\right\Vert ^{2}\mathrm{d}\mu_{X^{(L)}_{\infty}}}{\int_{X^{(L)}_{\infty}}f^{2}\mathrm{d}\mu_{X^{(L)}_{\infty}}} & =\frac{\int_{X^{(L)}_{\infty}}\left\Vert \nabla(\chi_{1}f)\right\Vert ^{2}\mathrm{d}\mu_{X^{(L)}_{\infty}}}{\int_{X^{(L)}_{\infty}}f^{2}\mathrm{d}\mu_{X^{(L)}_{\infty}}}\\
 & \leqslant\frac{1}{4}\left(1-\frac{\int_{X^{(L)}_{\infty}}(\chi_{2}f)^{2}\mathrm{d}\mu_{X^{(L)}_{\infty}}}{\int_{X^{(L)}_{\infty}}f^{2}\mathrm{d}\mu_{X^{(L)}_{\infty}}}\right)-\frac{d}{2}\\
 & =\frac{1}{4}\left(\frac{\int_{X^{(L)}_{\infty}}\left(\chi_{1}f\right)^{2}\mathrm{d}\mu_{X^{(L)}_{\infty}}}{\int_{X^{(L)}_{\infty}}f^{2}\mathrm{d}\mu_{X^{(L)}_{\infty}}}\right)-\frac{d}{2}.
\end{align*}
It follows that 
\[
\mathcal{R}\left(\chi_{1}f,\tilde{X}^{(L)}_{\infty}\right)\leqslant\frac{1}{4}-\frac{d}{2}\frac{\int_{X^{(L)}_{\infty}}f^{2}\mathrm{d}\mu_{X^{(L)}_{\infty}}}{\int_{X^{(L)}_{\infty}}\left(\chi_{1}f\right)^{2}\mathrm{d}\mu_{X^{(L)}_{\infty}}}\leqslant\frac{1}{4}-\frac{d}{2},
\]
since $\left(\chi_{1}f\right)^{2}=\left(f\sin\chi\right)^{2}\leqslant f^{2}$.
Using 
\[
\eta^{4}\left(\frac{1}{4}-\frac{d}{2}\right)\leqslant\frac{1}{4}-\frac{d}{4},
\]
we obtain $\mathcal{R}(F,X)\leqslant\frac{1}{4}-\frac{d}{4}$ and
so the result follows from Lemma \ref{lem:Minmax}. Indeed, if $X$
has less than $m$ eigenvalues below $\frac{1}{4}$, then by the second
part of Lemma \ref{lem:Minmax} we would have
\[
\frac{1}{4}=\inf\sigma_{\mathrm{ess}}(\Delta)\leq\sup\left\{ \mathcal{R}(F,X):F\in\mathrm{span}(\widetilde{(\chi_{1}\varphi_{i})\circ q},i=1,\ldots,m)\right\} ,
\]
which contradicts the above bound. The first part of Lemma \ref{lem:Minmax}
shows that the eigenvalues are smaller than $\frac{1}{4}-\frac{d}{4}$.
\end{proof}

\section{Proof of Theorem \ref{thm:main-thm}}

\label{sec:Proof-of-Theorem}

We will now prove Theorem \ref{thm:main-thm}. Let $X$ be a hyperbolic
surface with genus $g$ and $n$ cusps with $g<an$ for some $a>0$.
By Lemma \ref{thm:many-pants}, there then exist at least $c(a)n$
pairs of pants on $X$ with at least one cusp where

\[
c(a)\geqslant\frac{1}{64}\exp\left(-\frac{1}{2}-10\pi(2a+1)-2\exp\left(5\pi(2a+1)\right)\right),
\]
and the geodesic boundaries have length bounded by $10\pi(2a+1)$.

By Lemma \ref{lem:crit-exponent}, there exists a $\kappa(a)>0$ for
which any pair of pants with at least one cusp and boundaries with
lengths in $[0,20\pi(2a+1)]$ has critical exponent at least $\frac{1}{2}+\kappa(a)$.
Fix $d<\frac{\kappa(a)^{2}}{9}$ and fix $0<L<0.1$ to satisfy both
(\ref{eq:bilips-UB}) and (\ref{eq:dist-LB}). We will now find $b>0$
dependent on $a$ for which $X$ has at least $b(2g+n-2)$ small eigenvalues
smaller than $\frac{1}{4}-\frac{d}{4}$.

Our proof will follow from comparing the trace formula with the function
$f_{T}$ for different scales of $T$ independent of $g$ and $n$.
First we consider the case for $T=1$.
\begin{prop}
\label{prop:T=00003D1}Let $a>0$ be such that $g<an$. Let $X$ be
a (possibly disconnected) hyperbolic surface with (total) genus $g$
and $n$ cusps. Let $f_{1}$ and $h_{1}$ be defined as in equation
(\ref{eq:fn-choice}) and (\ref{eq:h_t}). Then, there exists $C(a)>0$
such that
\[
\sum_{\lambda_{j}\geqslant\frac{1}{4}}h_{1}(r_{j})+\frac{1}{4\pi}\int^{\infty}_{-\infty}h_{1}(r)\frac{-\varphi'}{\varphi}\left(\frac{1}{2}+ir\right)\mathrm{\mathrm{d}}r\leqslant C(a)\left(\sum_{\substack{\gamma\in\mathcal{P}(X)\\
\ell(\gamma)\leqslant\frac{L}{2}
}
}\log\left(\frac{1}{\ell(\gamma)}\right)+g+n\right).
\]
\end{prop}

\begin{proof}
We will assume that $X$ is connected because if this is not the case,
we can simply apply the following method to each connected component,
replacing $g$ and $n$ by the genus and number of cusps on each component,
and then summing the contributions.

We will analyze each term in the trace formula. First note that the
scattering matrix $\Phi(\frac{1}{2})$ is real and symmetric \cite[Theorem 6.6]{Iw2021}
and so its eigenvalues are $\pm1$, hence 
\[
\frac{h_{1}(0)}{4}\mathrm{Tr}\left(I-\Phi\left(\frac{1}{2}\right)\right)\leqslant\frac{h_{1}(0)}{2}n.
\]
Moreover, since $h_{1}$ is a Schwartz function and $|\psi(1+ir)|\leq\log(1+r^{2})+\mathrm{const}$
(see for example \cite[Formula 6.3.21]{Ab.St1968}), we easily obtain
\[
\left|\frac{\mathrm{Vol}(X)}{4\pi}\int^{\infty}_{-\infty}h_{1}(r)r\tanh(\pi r)\text{d}r-\frac{n}{2\pi}\int^{\infty}_{-\infty}h_{1}(r)\psi(1+ir)\mathrm{d}r\right|\leqslant C_{1}(g+n),
\]
for a constant $C_{1}>0$ dependent only upon $f_{1}$. Since $h_{1}\geqslant0$
on $\mathbb{R}\cup i\mathbb{R}$, 
\[
-\sum_{\lambda_{j}<\frac{1}{4}}h_{1}(r_{j})\leqslant0,
\]
and trivially, 
\[
nf_{1}(0)\log(2)\leqslant C_{2}n,
\]
for a constant $C_{2}>0$ depending only on $f_{1}$. Next, due to
the compact support of $f_{1}$,
\begin{align*}
\sum_{\gamma\in\mathcal{P}\left(X\right)}\sum^{\infty}_{k=1}\frac{\ell(\gamma)f_{1}(k\ell(\gamma))}{2\sinh\left(\frac{k\ell(\gamma)}{2}\right)} & \leqslant\|f_{1}\|_{\infty}\sum_{\substack{\gamma\in\mathcal{P}(X)\\
\ell(\gamma)\leqslant1
}
}\sum^{\left\lfloor \frac{1}{\ell(\gamma)}\right\rfloor }_{k=1}\frac{\ell(\gamma)}{2\sinh\left(\frac{k\ell(\gamma)}{2}\right)}\\
 & \leqslant\|f_{1}\|_{\infty}\sum_{\substack{\gamma\in\mathcal{P}(X)\\
\ell(\gamma)\leqslant1
}
}\sum^{\left\lfloor \frac{1}{\ell(\gamma)}\right\rfloor }_{k=1}\frac{1}{k}\\
 & \leqslant C_{3}\left(\sum_{\substack{\gamma\in\mathcal{P}(X)\\
\ell(\gamma)\leqslant\frac{L}{2}
}
}\log\left(\frac{1}{\ell(\gamma)}\right)+\#\{\gamma\in\mathcal{P}\left(X\right)\mid\frac{L}{2}<\ell(\gamma)\leqslant1\}\right)\\
 & \leqslant C_{3}\left(\sum_{\substack{\gamma\in\mathcal{P}(X)\\
\ell(\gamma)\leqslant\frac{L}{2}
}
}\log\left(\frac{1}{\ell(\gamma)}\right)+g+n\right),
\end{align*}
for a constant $C_{3}>0$ depending only on $f_{1}$ and $a$ and
changing between lines. Here, the second line follows from using $\sinh(x)\geqslant x$,
the third line follows from splitting up the lengths at $\frac{L}{2}$
and bounding the inner summation by a uniform constant depending only
on $L$ (and hence only on $a$) for the longer lengths, and comparing
the summation to an integral for the shorter ones. The final line
follows from the fact that the closed geodesics being shorter than
$1$ means they are simple and disjoint and hence form part of a pants
decomposition of the surface which has at most $3g+n-3$ curves. Put
together, Theorem \ref{thm:trace-formula} gives
\begin{align*}
\sum_{\lambda_{j}\geqslant\frac{1}{4}}h_{1}(r_{j})+\frac{1}{4\pi} & \int^{\infty}_{-\infty}h_{1}(r)\frac{-\varphi'}{\varphi}\left(\frac{1}{2}+ir\right)\text{d}r\\
 & =-\sum_{\lambda_{j}<\frac{1}{4}}h_{1}(r_{j})+\frac{\mathrm{Vol}(X)}{4\pi}\int^{\infty}_{-\infty}h_{1}(r)r\tanh(\pi r)\mathrm{d}r\\
 & +\sum_{\substack{\gamma\in\mathcal{P}\left(X\right)}
}\sum^{\infty}_{k=1}\frac{\ell(\gamma)f_{1}(k\ell(\gamma))}{2\sinh\left(\frac{k\ell(\gamma)}{2}\right)}+\frac{h_{1}(0)}{4}\mathrm{Tr}\left(I-\Phi\left(\frac{1}{2}\right)\right)\\
 & -nf_{1}(0)\log(2)-\frac{n}{2\pi}\int^{\infty}_{-\infty}h_{1}(r)\psi(1+ir)\mathrm{\mathrm{d}}r.\\
 & \leqslant C\left(\sum_{\substack{\gamma\in\mathcal{P}(X)\\
\ell(\gamma)\leqslant\frac{L}{2}
}
}\log\left(\frac{1}{\ell(\gamma)}\right)+g+n\right),
\end{align*}
for some constant $C>0$ depending only on $f_{1}$ and $a$. 
\end{proof}

\begin{proof}[Proof of Theorem \ref{thm:main-thm}]
 Let $X^{(L)}_{\infty}$ be the associated surface to $X$ by collapsing
down the closed geodesics of length at most $L$ as defined at the
start of Section \ref{sec:Proof-of-Theorem}. Then, at least $c_{1}(a)n$
pants on $X$ that were guaranteed by Lemma \ref{thm:many-pants}
will still be pants on $X^{(L)}_{\infty}$ except now some of their
boundaries will have been collapsed down to cusps. Thus, on $X^{L}_{\infty}$,
we have at least $c_{1}(a)n$ pairs of pants that are pairs of pants
with at least $1$ geodesic boundary and  thrice punctured spheres.
The remaining boundaries of these pants that have not been collapsed
are simple closed geodesics lying entirely in $\tilde{X}^{(L)}$. 

We can choose in the construction of the map $q$ for the completion
of the collection of closed geodesics shorter than $L$ to include
all of the boundaries in the $c_{1}(a)n$ pants. Then, under the map
$q$, the pairs of pants map to their corresponding (potentially)
collapsed pants on $X^{(L)}_{\infty}$ and their boundaries have length
at most $K[q]10\pi(2a+1)\leqslant20\pi(2a+1)$. Moreover, since the
non-collapsed boundaries are actually simple closed geodesics in $\tilde{X}^{(L)}_{\infty}$,
which has systole at least $\frac{1}{2}L$ by Lemma \ref{lem:curves-on-limit-surface},
their boundary lengths live inside $[\frac{1}{2}L,20\pi(2a+1)]$.
This will mean they are now amenable to the prime geodesic theorem
in Section \ref{sec:PGT}. 

Let $g_{\infty}$ and $n_{\infty}$ be the total genus and number
of cusps on the connected components of $X^{(L)}_{\infty}$. We will
assume that $X^{(L)}_{\infty}$ is connected since otherwise we can
apply the following argument to each connected component, replacing
$g$ and $n$ by the genus and number of cusps on each component,
and then summing the contributions. 

Let $T>1$. By Theorem \ref{thm:trace-formula} with $f_{T}$ for
$X^{(L)}_{\infty}$,
\begin{align*}
\sum_{\lambda_{j}<\frac{1}{4}}h_{T}(r_{j}) & =-\sum_{\lambda_{j}\geqslant\frac{1}{4}}h_{T}(r_{j})-\frac{1}{4\pi}\int^{\infty}_{-\infty}h_{T}(r)\frac{-\varphi'}{\varphi}\left(\frac{1}{2}+ir\right)\text{d}r\\
 & +\frac{\mathrm{Vol}(X^{(L)}_{\infty})}{4\pi}\int^{\infty}_{-\infty}h_{T}(r)r\tanh(\pi r)\mathrm{d}r\\
 & +\sum_{\substack{\gamma\in\mathcal{P}\left(X^{(L)}_{\infty}\right)}
}\sum^{\infty}_{k=1}\frac{\ell(\gamma)f_{T}(k\ell(\gamma))}{2\sinh\left(\frac{k\ell(\gamma)}{2}\right)}+\frac{h_{T}(0)}{4}\mathrm{Tr}\left(I-\Phi\left(\frac{1}{2}\right)\right)\\
 & -n_{\infty}f_{T}(0)\log(2)-\frac{n_{\infty}}{2\pi}\int^{\infty}_{-\infty}h_{T}(r)\psi(1+ir)\mathrm{\mathrm{d}}r.
\end{align*}
First, by (\ref{eq:h_t}) and the fact that $h_{1}$ is a Schwartz
function and non-negative on $\mathbb{R}$,
\[
\frac{\mathrm{Vol}\left(X^{(L)}_{\infty}\right)}{4\pi}\int^{\infty}_{-\infty}h_{T}(r)r\tanh(\pi r)\mathrm{d}r\leqslant A_{0}\left(\frac{g_{\infty}+n_{\infty}}{T}\right),
\]
for some constant $A_{0}>0$ depending only on $f_{1}$. Similarly,
using the additional estimate $|\psi(1+ir)|\leqslant\log(1+r^{2})+\mathrm{const}$
as in the proof of Proposition \ref{prop:T=00003D1}, 
\[
\left|\frac{n_{\infty}}{2\pi}\int^{\infty}_{-\infty}h_{T}(r)\psi(1+ir)\mathrm{\mathrm{d}}r\right|\leqslant A_{1}n_{\infty},
\]
and moreover, since $h_{T}(0)\leqslant\|h_{1}\|_{\infty}T$, 
\[
\left|\frac{h_{T}(0)}{4}\mathrm{Tr}\left(I-\Phi\left(\frac{1}{2}\right)\right)\right|\leqslant A_{2}Tn_{\infty},
\]
for some constants $A_{1},A_{2}>0$ depending only upon $f_{1}$.
Trivially we also have
\[
n_{\infty}f_{T}(0)\log(2)\leqslant A_{3}n_{\infty},
\]
for some constant $A_{3}>0$ depending only upon $f_{1}$.

We now turn our attention to 
\[
-\sum_{\lambda_{j}\geqslant\frac{1}{4}}h_{T}(r_{j})-\frac{1}{4\pi}\int^{\infty}_{-\infty}h_{T}(r)\frac{-\varphi'}{\varphi}\left(\frac{1}{2}+ir\right)\text{d}r.
\]
By property 3 of the function $f_{1}$ using that $r_{j}\geqslant0$
since $\lambda_{j}\geqslant\frac{1}{4}$, we have 
\[
\sum_{\lambda_{j}\geqslant\frac{1}{4}}h_{T}(r_{j})=T\sum_{\lambda_{j}\geqslant\frac{1}{4}}h_{1}(Tr_{j})\leqslant T\sum_{\lambda_{j}\geqslant\frac{1}{4}}h_{1}(r_{j}).
\]
 Also, following \cite[(11.9)]{Iw2021},
\[
\frac{-\varphi'}{\varphi}\left(\frac{1}{2}+ir\right)=2\log b_{1}+S_{+}(r)+S_{-}(r),
\]
where
\begin{align*}
S_{+}(r) & =\sum_{\substack{j:s_{j}\text{ is a pole of }\varphi\\
s_{j}=\beta_{j}+i\sigma_{j},\ \beta_{j}<\frac{1}{2}
}
}\frac{1-2\beta_{j}}{\left(\frac{1}{2}-\beta_{j}\right)^{2}+\left(r-\sigma_{j}\right)^{2}}\geqslant0,\\
S_{-}(r) & =\sum_{\substack{j:\beta_{j}\text{ is a pole of }\varphi\\
\frac{1}{2}<\beta_{j}\leqslant1
}
}\frac{1-2\beta_{j}}{r^{2}+\left(\frac{1}{2}-\beta_{j}\right)^{2}}\leqslant0,
\end{align*}
and $b_{1}$ is the first coefficient in (\ref{eq:Dirichlet-determinant})
with $b_{1}\geqslant1$ c.f. Lemma \ref{lem:dirichlet-coef}. Thus,
\begin{align*}
\int^{\infty}_{-\infty}h_{T}(r)\frac{-\varphi'}{\varphi}\left(\frac{1}{2}+ir\right)\text{d}r & \leqslant\int^{\infty}_{-\infty}h_{T}\left(r\right)\left(2\log b_{1}+S_{+}(r)\right)\mathrm{d}r\\
 & \leqslant\int^{\infty}_{-\infty}Th_{1}\left(Tr\right)\left(2\log b_{1}+S_{+}(r)\right)\mathrm{d}r\\
 & \leqslant T\int^{\infty}_{-\infty}h_{1}\left(r\right)\left(2\log b_{1}+S_{+}(r)\right)\mathrm{d}r
\end{align*}
again, by properties 2 and 3 of the function $f_{1}$. It follows
from Proposition \ref{prop:T=00003D1} that there exists a constant
$A_{4}>0$ depending only upon $f_{1}$ such that

\begin{align*}
-\sum_{\lambda_{j}\geqslant\frac{1}{4}}h_{T}(r_{j}) & -\frac{1}{4\pi}\int^{\infty}_{-\infty}h_{T}(r)\frac{-\varphi'}{\varphi}\left(\frac{1}{2}+ir\right)\text{d}r\\
 & \geqslant-T\sum_{\lambda_{j}\geqslant\frac{1}{4}}h_{1}(r_{j})-\frac{T}{4\pi}\int^{\infty}_{-\infty}h_{1}\left(r\right)\left(2\log b_{1}+S_{+}(r)\right)\mathrm{d}r\\
 & \geqslant\frac{T}{4\pi}\int^{\infty}_{-\infty}h_{1}\left(r\right)S_{-}(r)\mathrm{d}r-A_{4}T\left(g_{\infty}+n_{\infty}\right).
\end{align*}
The summation over the closed geodesics from Proposition \ref{prop:T=00003D1}
vanished because $\mathrm{sys}(X^{(L)}_{\infty})>\frac{1}{2}L$. The
integral
\[
\int^{\infty}_{-\infty}h_{1}\left(r\right)S_{-}(r)\mathrm{d}r
\]
can be controlled by an integration by parts. Note that 
\[
\int^{\infty}_{-\infty}S_{-}(r)h_{1}(r)\mathrm{d}r=\sum_{\substack{j:\beta_{j}\text{ is a pole of }\varphi\\
\frac{1}{2}<\beta_{j}\leqslant1
}
}\int^{\infty}_{-\infty}\frac{1-2\beta_{j}}{r^{2}+\left(\frac{1}{2}-\beta_{j}\right)^{2}}h_{1}(r)\mathrm{d}r,
\]
where interchanging summation and integration is justified since the
sum has only finitely many terms \cite{Ot.Ro09}. Then for any $\beta_{j}>\frac{1}{2}$,
\begin{align*}
\left|\int^{\infty}_{-\infty}\frac{1-2\beta_{j}}{r^{2}+\left(\frac{1}{2}-\beta_{j}\right)^{2}}h_{1}(r)\mathrm{d}r\right| & =\left|2\int^{\infty}_{-\infty}h_{1}'(r)\arctan\left(\frac{r}{\frac{1}{2}-\beta_{j}}\right)\mathrm{d}r\right|\\
 & \leqslant\pi\int^{\infty}_{-\infty}\left|h_{1}'(r)\right|\mathrm{d}r.
\end{align*}
It follows that 
\[
\left|\int^{\infty}_{-\infty}S_{-}(r)h_{1}(r)\mathrm{d}r\right|\leqslant A_{5}(g_{\infty}+n_{\infty}),
\]
for a constant $A_{5}>0$ dependent only upon $f_{1}$. Put together,
along with the fact that $\mathrm{supp}(f_{T})\subseteq[-T,T]$, there
exists a constant $A>0$ depending only upon $f_{1}$ such that

\begin{align}
\sum_{\lambda_{j}<\frac{1}{4}}h_{T}(r_{j}) & \geqslant\underbrace{\sum_{\substack{\gamma\in\mathcal{P}\left(X^{(L)}_{\infty}\right)\\
\ell(\gamma)\leqslant T
}
}\sum^{\infty}_{k=1}\frac{\ell(\gamma)f_{T}(k\ell(\gamma))}{2\sinh\left(\frac{k\ell(\gamma)}{2}\right)}}_{(a)}-AT\left(g_{\infty}+n_{\infty}\right).\label{eq:LB}
\end{align}
Since the summand of (a) is non-negative and $2\sinh\left(\frac{\ell(\gamma)}{2}\right)\leqslant e^{\frac{\ell(\gamma)}{2}}$
we have 
\begin{align*}
\sum_{\substack{\gamma\in\mathcal{P}\left(X^{(L)}_{\infty}\right)\\
\ell(\gamma)\leqslant T
}
}\sum^{\infty}_{k=1}\frac{\ell(\gamma)f_{T}(k\ell(\gamma))}{2\sinh\left(\frac{k\ell(\gamma)}{2}\right)} & \geqslant\sum_{\substack{\gamma\in\mathcal{P}\left(X^{(L)}_{\infty}\right)\\
1<\ell(\gamma)\leqslant T
}
}\sum^{\infty}_{k=1}\frac{\ell(\gamma)f_{T}(k\ell(\gamma))}{2\sinh\left(\frac{k\ell(\gamma)}{2}\right)}\\
 & \geqslant\sum_{\substack{\gamma\in\mathcal{P}\left(X^{(L)}_{\infty}\right)\\
1<\ell(\gamma)\leqslant T
}
}\ell(\gamma)e^{-\frac{\ell(\gamma)}{2}}f_{1}\left(\frac{\ell(\gamma)}{T}\right).
\end{align*}
Up to now, if we were instead working with a connected component of
$\tilde{X}^{(L)}_{\infty}$ we would have obtained the exact same
bounds but with $g_{\infty}$ and $n_{\infty}$ replaced by the respective
genera and number of cusps on the component. Summing the contributions
and using that the genera and number of cusps will sum to $g_{\infty}$
and $n_{\infty}$ then gives the exact same outcome. 

We now exploit the existence of the embedded pants in $X^{(L)}_{\infty}$.

Let $P$ be such a pair of pants, then any closed geodesic inside
$P$ gives rise to a unique closed geodesic of the same length in
the surface $X^{(L)}_{\infty}$. A lower bound for (a) (using non-negativity
of $f_{T}$) arises from a lower bound on the number of geodesics
in each of these pairs of pants. Let $T$ be sufficiently large (independent
of the surface) so that by Theorems \ref{thm:Naud-PGT} and \ref{thm:thrice-punc-sphere},
and equation (\ref{eq:log-int}), since the pants have their boundaries
in $[\frac{1}{2}L,20\pi(2a+1)]$, there exists a constant $B>0$ (depending
only upon $a$) such that 

\begin{align*}
\#\left\{ \gamma\in\mathcal{P}\left(P\right)\mid\frac{T}{6}\leqslant\ell(\gamma)\leqslant\frac{T}{3}\right\}  & \geqslant\frac{e^{\delta\left(P\right)\frac{T}{3}}}{\delta(P)\frac{T}{3}}-\frac{e^{\delta\left(P\right)\frac{T}{6}}}{\delta(P)\frac{T}{6}}\left(1+\frac{12}{\delta(P)T}\right)\\
 & \thinspace\thinspace\thinspace-B\left(e^{\left(\frac{\delta\left(P\right)}{6}+\frac{1}{12}\right)T}+e^{\left(\frac{\delta\left(P\right)}{12}+\frac{1}{24}\right)T}\right)\\
 & \geqslant\frac{e^{\delta\left(P\right)\frac{T}{3}}}{T}-2Be^{\left(\frac{\delta\left(P\right)}{6}+\frac{1}{12}\right)T}.
\end{align*}
Indeed, since $\delta(P)>\frac{1}{2}$, we can pick $T$ such that
$12e^{-\frac{T}{12}}\left(1+\frac{24}{T}\right)\leqslant2$ for which
$T>12\log(12)$ works. Then, for sufficiently large $T$ (depending
only upon $f_{1}$ and $a$) we have
\begin{align}
\sum_{\substack{\gamma\in\mathcal{P}\left(P\right)\\
1<\ell(\gamma)\leqslant T
}
}\ell(\gamma)e^{-\frac{l(\gamma)}{2}}f_{1}\left(\frac{\ell(\gamma)}{T}\right) & \geqslant C\sum_{\substack{\gamma\in\mathcal{P}\left(P\right)\\
\frac{T}{6}\leqslant\ell(\gamma)\leqslant\frac{T}{3}
}
}\ell(\gamma)e^{-\frac{l(\gamma)}{2}}\nonumber \\
 & \geqslant\frac{C}{3}e^{-\frac{T}{6}}\left(e^{\delta\left(P\right)\frac{T}{3}}-2BTe^{\left(\frac{\delta\left(P\right)}{6}+\frac{1}{12}\right)T}\right)\label{eq:pants-geo-count}\\
 & \geqslant\frac{C}{6}e^{\left(\delta\left(P\right)-\frac{1}{2}\right)\frac{T}{3}},\nonumber 
\end{align}
where $C=\inf_{x\in\left[\frac{1}{6},\frac{1}{3}\right]}f_{1}(x)>0$
and we have used non-negativity of the summand to restrict to geodesics
with lengths in $\left[\frac{T}{6},\frac{T}{3}\right]$. Recall that
$\kappa=\kappa(a)>0$ is from Lemma \ref{lem:crit-exponent} such
that $\delta\left(P\right)>\frac{1}{2}+\kappa$ uniformly for all
of the pants $P$ embedded as above inside $X^{(L)}_{\infty}$. Then,
\begin{equation}
\sum_{\substack{\gamma\in\mathcal{P}\left(X^{(L)}_{\infty}\right)\\
\ell(\gamma)\leqslant T
}
}\sum^{\infty}_{k=1}\frac{\ell(\gamma)f_{T}(k\ell(\gamma))}{2\sinh\left(\frac{k\ell(\gamma)}{2}\right)}\geqslant Dn_{\infty}e^{\frac{1}{3}\kappa T},\label{eq:geom-lb}
\end{equation}
for a constant $D>0$ that depends only upon $f_{1}$ and $a$. The
lower bound with $n_{\infty}$ and some exponential growth in $T$
will be paramount in establishing that this dominates any arising
error terms.

On the other hand, basic estimates with $h_{T}$ show that since $0<d<\kappa^{2}$
(recall that we fixed $d$ with this property at the beginning of
this section),
\begin{align*}
\sum_{\lambda_{j}<\frac{1}{4}}h_{T}(r_{j}) & =\sum_{\lambda_{j}<\frac{1}{4}-d}h_{T}(r_{j})+\sum_{\frac{1}{4}-d\leqslant\lambda_{j}<\frac{1}{4}}h_{T}(r_{j})\\
 & \leqslant\#\left\{ j:\lambda_{j}<\frac{1}{4}-d\right\} \|f_{1}\|_{1}Te^{\frac{T}{2}}+\|f_{1}\|_{1}Te^{T\sqrt{d}}\#\left\{ j:\frac{1}{4}-d\leqslant\lambda_{j}<\frac{1}{4}\right\} .
\end{align*}
Thus putting this together with equations (\ref{eq:LB}) and (\ref{eq:geom-lb}),
we obtain a constant $\tilde{A}>0$ that depends only on $a$ and
$f_{1}$ such that
\begin{align*}
\#\left\{ j:\lambda_{j}<\frac{1}{4}-d\right\}  & \geqslant\frac{D}{\|f_{1}\|_{1}T}n_{\infty}e^{T\left(\frac{\kappa}{3}-\frac{1}{2}\right)}-\tilde{A}n_{\infty}e^{T\left(\sqrt{d}-\frac{1}{2}\right)}.
\end{align*}
Since $\sqrt{d}<\frac{\kappa}{3}$, we can choose $T$ large enough
depending only upon $a$ and $f_{1}$ so that 
\[
\tilde{A}n_{\infty}e^{T\left(\sqrt{d}-\frac{1}{2}\right)}<\frac{D}{2\|f_{1}\|_{1}T}n_{\infty}e^{T\left(\frac{\kappa}{3}-\frac{1}{2}\right)}.
\]
It follows that 
\[
\#\left\{ j:\lambda_{j}<\frac{1}{4}-d\right\} \geqslant\frac{D}{2\|f_{1}\|_{1}T}n_{\infty}e^{T\left(\frac{\kappa}{3}-\frac{1}{2}\right)}.
\]
Since $g_{\infty}\leqslant g<an\leqslant an_{\infty}$ (because collapsing
geodesics can only reduce genus and increase the number of cusps),
the existence of a $b>0$ depending only upon $a$ and $f_{1}$ for
which $X^{(L)}_{\infty}$ has at least $b\mathrm{Vol}(X^{(L)}_{\infty})=b\mathrm{Vol}(X)\geqslant2\pi b(2g+n-2)$
eigenvalues smaller than $\frac{1}{4}-d$ now follows. Since $L$
was chosen to satisfy (\ref{eq:bilips-UB}) and (\ref{eq:dist-LB})
at the start of the section, we can apply Proposition \ref{prop:passing-evalues}
to obtain at least $2\pi b(2g+n-2)$ on $X$ that are smaller than
$\frac{1}{4}-\frac{d}{4}$ as required.
\end{proof}

\section*{Acknowledgments}

JT was supported by funding from the Leverhulme Trust through a Leverhulme
Early Career Fellowship (grant number ECF-2024-440) and while some
of this work was being completed was also funded by the European Research
Council (ERC) under the European Union’s Horizon 2020 research and
innovation programme (grant agreement No 949143). We thank Irving
Calderón, Michael Magee and Sugata Mondal for conversations about
this work. We also thank the referees for their comments and corrections
in particular to the proof of Proposition \ref{prop:passing-evalues}.

\bibliographystyle{amsalpha}
\bibliography{manyeigenvalues}

@book{Ab.St1968,
  title = {Handbook of {{Mathematical Functions}} with {{Formulas}}, {{Graphs}}, and {{Mathematical Tables}}},
  author = {Abramowitz, Milton and Stegun, Irene A.},
  year = {1968},
  publisher = {U.S. Government Printing Office},
  googlebooks = {ZboM5tOFWtsC},
  isbn = {978-0-16-000202-1},
  langid = {english}
}

@article{Ba.Ma.Mo17,
  title = {Small Eigenvalues of Surfaces of Finite Type},
  author = {Ballmann, Werner and Matthiesen, Henrik and Mondal, Sugata},
  year = {2017},
  journal = {Compositio Mathematica},
  volume = {153},
  number = {8},
  pages = {1747--1768},
  publisher = {London Mathematical Society},
  doi = {10.1112/S0010437X17007291}
}

@article{Ba.Po.24,
title={Small eigenvalues of Schrödinger operators over geometrically finite manifolds},
author = {Ballman, Werner and Polymerakis, Panagiotis},
year = {2024},
journal = {Mathematical Research Letters},
volume= {31},
pages = {1--24}
}

@article{Be.1968,
  title = {The Exponent of Convergence of {P}oincaré Series},
  author = {Beardon, Alan},
  year = {1968},
  journal = {Proceedings of the London Mathematical Society},
  volume = {04},
  number = {03},
  pages = {461--482},
  doi = {10.1112/plms/s3-18.3.461},

}

@article{Bu.1977,
  title = {Riemannsche {F}lächen mit {E}igenwerten in
(0,1/4)},
  author = {Buser, Peter},
  year = {1977},
  journal = {Comment. Math. Helv.},
  volume = {52},
  number = {1},
  pages = {25--34},
}

@book{Bu2010,
  title = {Geometry and Spectra of Compact {{Riemann}} Surfaces},
  author = {Buser, Peter},
  year = {2010},
  series = {Modern Birkh{\"a}user Classics},
  pages = {xvi+454},
  publisher = {Birkh{\"a}user Boston, Inc., Boston, MA},
  doi = {10.1007/978-0-8176-4992-0},
  isbn = {978-0-8176-4991-3},
  mrclass = {58J50 (30F10 32G15 58J53)}
}

@article{Bu.Ma.Mu.Si2014,
  title={Quasiconformal embeddings of Y-pieces},
  author={Buser, Peter and Makover, Eran and Muetzel, Bjoern and Silhol, Robert},
  journal={Computational Methods and Function Theory},
  volume={14},
  number={2},
  pages={431--452},
  year={2014},
  publisher={Springer}
}

@article{Co.Co.1989,
  title = {Les valeurs propres inférieures à 1/4 des surfaces de {R}iemann de petit rayon d’injectivité},
  author = {Colbois, B. and Courtois, G.},
  year = {1989}, 
  journal = {Commentarii Mathematici Helvetici},
  volume = {64},
  pages = {349--362}, 
}

@article{Hi.Th.23,
  title = {Large-n asymptotics for {Weil-Petersson} volumes of moduli spaces of bordered hyperbolic surfaces},
  author = {Hide, Will and Thomas, Joe},
  year = {2023},
  journal = {arXiv:2312.11412},
  volume = {preprint}
}

@book{Iw2021,
  title = {Spectral {{Methods}} of {{Automorphic Forms}}: {{Second Edition}}},
  shorttitle = {Spectral {{Methods}} of {{Automorphic Forms}}},
  author = {Iwaniec, Henryk},
  year = {2021},
  month = nov,
  publisher = {American Mathematical Society, Revista Matem{\'a}tica Iberoamericana (RMI), Madrid, Spain},
  googlebooks = {OVxQEAAAQBAJ},
  isbn = {978-1-4704-6622-0},
  langid = {english}
}

@article{Li20,
title= {Finiteness of Small Eigenvalues of Geometrically Finite Rank one Locally Symmetric Manifolds},
author = {Jialun Li},
journal = {Mathematical Research Letters},
volume={27},
issue={2},
pages = {465--500}
}

@article{Ma.Na.Pu2022,
  title = {A Random Cover of a Compact Hyperbolic Surface Has Relative Spectral Gap {$\frac{3}{16}-\varepsilon$}},
  author = {Magee, Michael and Naud, Fr{\'e}d{\'e}ric and Puder, Doron},
  year = {2022},
  journal = {Geometric and Functional Analysis},
  volume = {32},
  number = {3},
  pages = {595--661}
}

@article{McK1972,
  title = {Selberg’s trace formula as applied to a
compact {R}iemann surface},
  author = {Mckean, H. P.},
  year = {1972},
  journal = {Comm. Pure Appl. Math},
  volume = {25},
  pages = {225–-246},
 
}

@article{McM1998,
  title = {Hausdorff Dimension and Conformal Dynamics {{III}}: {{Computation}} of Dimension},
  author = {McMullen, Curtis T.},
  year = {1998},
  journal = {American journal of mathematics},
  volume = {120},
  number = {4},
  pages = {691--721}
}

@article{Na2005,
  title = {Precise Asymptotics of the Length Spectrum for Finite-Geometry {{Riemann}} Surfaces},
  author = {Naud, Fr{\'e}d{\'e}ric},
  year = {2005},
  journal = {International Mathematics Research Notices},
  volume = {2005},
  number = {5},
  pages = {299--310},
  issn = {1687-0247},
  doi = {10.1155/IMRN.2005.299},
  urldate = {2024-09-16}
}

@article{Ot.Ro09,
  title = {Pour Toute Surface Hyperbolique de Genre {$g$}, {$\lambda_{2g-2}>\frac{1}{4}$}},
  author = {Otal, Jean-Pierre and Rosas, Eulalio},
  year = {2009},
  journal = {Duke Mathematical Journal},
  volume = {150},
  number = {1},
  pages = {101--115},
  publisher = {Duke University Press},
  doi = {10.1215/00127094-2009-048}
}

@article{Ra.74,
author = {Randol, Burton},
title = {{Small eigenvalues of the Laplace operator on compact Riemann surfaces}},
volume = {80},
journal = {Bulletin of the American Mathematical Society},
number = {5},
publisher = {American Mathematical Society},
pages = {996 -- 1000},
year = {1974},
}

@article{Po.Vy.17,
title= {Critical points for the Hausdorff dimension of pairs of pants},
author = {Pollicott, Mark and Vytnova, Polina},
year= {2017},
journal= {Groups, Geometry and Dynamics},
volume= {11},
issue= {4},
Pages = {1497--1519},
}

@article{Sc1990,
  title = {Small eigenvalues on {Y}-pieces and on {R}iemann surfaces},
  author = {Schmutz, Paul},
  year = {1990},
  journal = {Comment. Math. Helv},
  volume = {65},
  number = {4},
  pages = {603--614}
}

@article{Sc1991,
  title = {Small eigenvalues on {R}iemann surfaces
of genus 2},
  author = {Schmutz, Paul},
  year = {1991},
  journal = {Invent. math.},
  volume = {106},
  number = {1},
  pages = {121--138}
}

@article{Se1965,
  title = {On the estimation of {F}ourier coefficients of modular forms},
  author = {Selberg, Atle},
  year = {1965},
  journal = {Proc.
Sympos. Pure Math.},
  volume = {Vol. VIII},
  publisher = {Amer. Math. Soc},
  pages = {1--15}
}

@article{Sh.Wu22,
  title = {Arbitrarily Small Spectral Gaps for Random Hyperbolic Surfaces with Many Cusps},
  author = {Shen, Yang and Wu, Yunhui},
  year = {2022},
  journal = {arXiv:2203.15681},
  eprint = {2203.15681},
  archiveprefix = {arXiv}
}

@book{Sm12,
  title = {Unbounded Self-Adjoint Operators on {H}ilbert Space},
  author = {Schm{\"u}dgen, Konrad},
  year = {2012},
  series = {Graduate Texts in Mathematics},
  pages = {432},
  publisher = {Springer Dordrecht}
}

@book{Th2022,
  title = {The {{Geometry}} and {{Topology}} of {{Three-Manifolds}}: {{With}} a {{Preface}} by {{Steven P}}. {{Kerckhoff}}},
  shorttitle = {The {{Geometry}} and {{Topology}} of {{Three-Manifolds}}},
  author = {Thurston, William P.},
  year = {2022},
  month = jul,
  publisher = {American Mathematical Society},
}

@article{Tl2022,
  title = {Bump Functions with Monotone {{Fourier}} Transforms Satisfying Decay Bounds},
  author = {Tlas, T.},
  year = {2022},
  month = jun,
  journal = {Journal of Approximation Theory},
  volume = {278},
  pages = {105742},
  issn = {0021-9045},
  doi = {10.1016/j.jat.2022.105742},
  urldate = {2024-09-05}
}

@article{Zo1987,
  title = {Small Eigenvalues of Automorphic {{Laplacians}} in Spaces of Parabolic Forms},
  author = {Zograf, Peter},
  year = {1987},
  month = jan,
  journal = {Journal of Soviet Mathematics},
  volume = {36},
  number = {1},
  pages = {106--114},
  issn = {1573-8795},
  doi = {10.1007/BF01104976},
  urldate = {2023-09-10}
}

\noindent Will Hide, \\
Mathematical Institute,\\
University of Oxford, \\
Andrew Wiles Building, OX2 6GG Oxford,\\
United Kingdom

\noindent\texttt{william.hide@maths.ox.ac.uk}~\\
\texttt{}~\\

\noindent Joe Thomas, \\
Department of Mathematical Sciences,\\
Durham University, \\
Lower Mountjoy, DH1 3LE Durham,\\
United Kingdom

\noindent\texttt{joe.thomas@durham.ac.uk}

\end{document}